\documentclass{article}

\usepackage{amsfonts, amssymb, amsmath, amsthm, epsf, cite}

\numberwithin{figure}{section}

\theoremstyle{plain}
\newtheorem{theorem}{Theorem}[section]
\newtheorem{lemma}{Lemma}[section]

\newtheorem{condition}{Condition}[section]

\theoremstyle{definition}
\newtheorem{definition}{Definition}[section]

\newtheorem{remark}{Remark}[section]

\numberwithin{equation}{section}

\renewcommand{\Im}{{\rm Im\,}}
\renewcommand{\Re}{{\rm Re\,}}

\newcommand{\ind}{{\rm ind\,}}
\newcommand{\supp}{{\rm supp\,}}

\renewcommand{\ker}{{\rm ker\,}}
\renewcommand{\dim}{{\rm dim\,}}

\newcommand{\Orb}{{\rm Orb}}

\newcommand{\dist}{{\rm dist}}
\renewcommand{\phi}{{\varphi}}

\newcommand{\cK}{{\mathcal K}}

\newcommand{\cP}{{\mathcal P}}
\newcommand{\cL}{{\mathcal L}}
\newcommand{\cB}{{\mathcal B}}

\newcommand{\cO}{{\mathcal O}}

\newcommand{\cH}{{\mathcal H}}
\newcommand{\cS}{{\mathcal S}}

\newcommand{\cW}{{\mathcal W}}

\newcommand{\pG}{{\partial G}}

\newcommand{\oG}{{\overline G}}

\newcommand{\bP}{{\mathbf P}}
\newcommand{\bB}{{\mathbf B}}

\newcommand{\bS}{{\mathbf S}}

\newcommand{\bW}{{\mathbf W}}
\newcommand{\bu}{{\mathbf u}}
\newcommand{\bv}{{\mathbf v}}

\newcommand{\mimu}{{2m- m_{i\mu}-1/2}}
\newcommand{\bm}{{2m-{\mathbf m}-1/2}}
\newcommand{\mjsigma}{{2m- m_{j\sigma\mu}-1/2}}

\newcommand{\bbN}{{\mathbb N}}
\newcommand{\bbC}{{\mathbb C}}

\title{Smoothness of generalized solutions for higher-order elliptic equations with nonlocal boundary conditions}

\author{Pavel Gurevich~\thanks{This research was supported by
  the Alexander von
Humboldt Foundation and RFBR (project No.~07-01-00268).}}

\date{}

\begin{document}

\maketitle

\begin{abstract}
Smoothness of generalized solutions for higher-order elliptic
equations with nonlocal boundary conditions is studied in plane
domains. Necessary and sufficient conditions upon the right-hand
side of the problem and nonlocal operators under which the
generalized solutions possess an appropriate smoothness are
established.
\end{abstract}

\section{Introduction}

In 1932, Carleman~\cite{Carleman} considered the problem of finding a harmonic function,
in a plane bounded domain, satisfying a nonlocal condition which connects the values of
the unknown function at different points of the boundary. Further investigation of
elliptic problems with transformations mapping a boundary onto itself as well as with
abstract nonlocal conditions has been carried out by Vishik~\cite{Vishik},
Browder~\cite{Browder}, Beals~\cite{Beals}, Antonevich~\cite{Antonevich}, and others.

In 1969, Bitsadze and Samarskii~\cite{BitzSam} considered the following nonlocal problem
arising in the plasma theory: to find a function $u(y_1, y_2)$ harmonic on the
rectangular $G=\{y\in\mathbb R^2: -1<y_1<1,\ 0<y_2<1\}$, continuous on $\overline{G}$,
and satisfying the relations
 \begin{align*}
  u(y_1, 0)&=f_1(y_1),\quad u(y_1, 1)=f_2(y_1),\quad -1<y_1<1,\\
  u(-1, y_2)&=f_3(y_2),\quad u(1, y_2)=u(0, y_2),\quad 0<y_2<1,
 \end{align*}
where $f_1, f_2, f_3$ are given continuous functions. This problem was solved
in~\cite{BitzSam} by reducing it to a Fredholm integral equation and  using the maximum
principle. For arbitrary domains and  general nonlocal conditions, such a problem was
formulated as an unsolved one (see also~\cite{Sam, krall}). Different generalizations of
nonlocal problems with transformations mapping the boundary inside the closure of a
domain were studied by many authors~\cite{ZhEid,RSh,Kishk,GM}.

The most complete theory for elliptic equations of order $2m$ with general nonlocal
conditions was developed by Skubachevskii and his
students~\cite{SkMs83,SkMs86,SkDu90,SkDu91,SkJMAA,KovSk,SkBook,GurRJMP03}: a
classification with respect to types of nonlocal conditions was suggested, the Fredholm
solvability in the corresponding spaces was investigated, and asymptotics of solutions
near special conjugation points was obtained.

Note that, besides the plasma theory,   nonlocal elliptic problems have interesting
applications to biophysics and theory of diffusion
processes~\cite{Feller,Ventsel,SatoUeno,Taira,GalSkub}, control
theory~\cite{BensLions,Amann}, theory of functional differential equations,
mechanics~\cite{SkBook}, and so on.

The most difficult situation in the theory of nonlocal problems is that where the support
of nonlocal terms can intersect the boundary of a domain.  In this case, solutions of
nonlocal problems can have power-law singularities near some points of the boundary even
if the right-hand side is infinitely differentiable and the boundary is infinitely
smooth~\cite{SkMs86, SkRJMP, GurAdvDiffEq}. This gives rise to the question of
distinguishing some classes of nonlocal problems whose  solutions are sufficiently
smooth, provided that the right-hand side of the problem is smooth. Until now, this issue
was studied only for nonlocal perturbations of the Dirichlet problem for second-order
elliptic equations~\cite{SkRJMP, GurAdvDiffEq}.

In the present paper, we investigate  the smoothness of solutions for   {\it elliptic
equations of higher order with general nonlocal conditions in plane domains.\/} Unlike
the theory of elliptic problems in nonsmooth domains,   the violation of smoothness of
solutions for nonlocal problems is connected not only with the fact that the boundary may
contain singular points but rather with the presence of nonlocal terms in the boundary
conditions.

We illustrate some of the occurring phenomena with the following example. Let $\partial
G=\Gamma_1\cup\Gamma_2\cup\{g,h\}$, where $\Gamma_i$ are open (in the topology of
$\partial G$) $C^\infty$ curves; $g,h$ are the end points of the curves
$\overline{\Gamma_1}$ and $\overline{\Gamma_2}$. Suppose that the domain $G$ is the plane
angle of opening $\pi$ in some neighborhood of each of the points $g$ and $h$. We
deliberately take a smooth domain   to illustrate how the nonlocal terms can affect the
smoothness of solutions. Consider the following  problem in the domain $G$:
\begin{equation}\label{eqIntroPinG}
 \Delta u=f_0(y)\quad (y\in G),
\end{equation}
\begin{equation}\label{eqIntroBinG}
\begin{aligned}
&u|_{\Gamma_1}+b_1(y) u (\Omega_{1}(y) ) |_{\Gamma_1}+a(y) u (\Omega(y) )
|_{\Gamma_1}=f_1(y) & &
(y\in\Gamma_1),\\
&u|_{\Gamma_2}+b_2(y) u (\Omega_{2}(y) ) |_{\Gamma_2}=f_2(y) & &
 (y\in\Gamma_2).
\end{aligned}
\end{equation}
Here $b_1$, $b_2$, and $a$ are real-valued $C^\infty$ functions; $\Omega_i$ ($\Omega$)
are $C^\infty$ diffeomorphisms taking some neighborhood ${\mathcal O}_i$ (${\mathcal
O}_1$) of the curve $\Gamma_i$ ($\Gamma_1$) onto the set $\Omega_i({\mathcal O}_i)$
($\Omega({\mathcal O}_1)$) in such a way that $\Omega_i(\Gamma_i)\subset G$,
$\Omega_i(g)=g$, $\Omega_i(h)=h$, and the transformation $\Omega_i$, near the points $g,
h$, is the rotation of the boundary $\Gamma_i$ through the angle $\pi/2$ inwards the
domain $G$ (respectively, $\Omega(\Gamma_1)\subset G$,
$\overline{\Omega(\Gamma_1)}\cap\{g,h\}=\varnothing$, and the approach of the curve
$\Omega(\overline{\Gamma_1})$ to the boundary $\partial G$ can be arbitrary,
cf.~\cite{SkMs86, SkDu91}), see Fig.~\ref{figEx1}.
\begin{figure}[ht]
{ \hfill\epsfxsize100mm\epsfbox{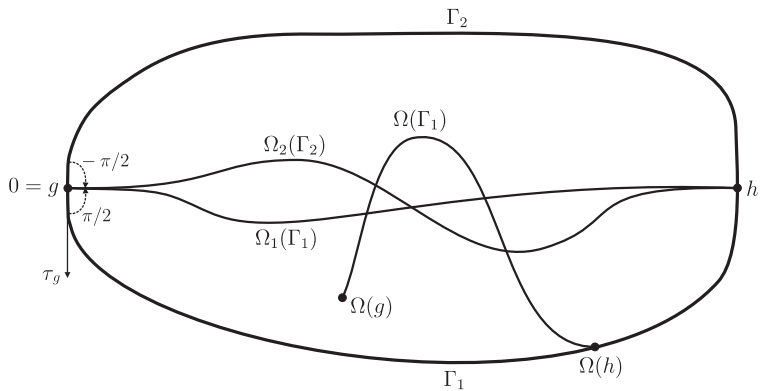}\hfill\ } \caption{Domain $G$ with boundary
$\partial G=\Gamma_1\cup\Gamma_2\cup\{g,h\}$.}
   \label{figEx1}
\end{figure}

We say that $g$ and $h$ are the {\em points of conjugation of nonlocal conditions}
because they divide the curves on which different nonlocal conditions are set. The
closure of the set
$$
\bigcup_{i=1,2}\{y\in\Omega_i(\Gamma_i):\
b_i(\Omega_i^{-1}(y))\ne0\}\cup\{y\in\Omega(\Gamma_1):\ a(\Omega^{-1}(y))\ne0\}
$$
is referred to as the {\em support of nonlocal terms}.

Denote by $W^k(G)=W^k_2(G)$ the Sobolev space. We say that a function $u\in W^1(G)$ is a
{\em generalized solution} of problem~\eqref{eqIntroPinG}, \eqref{eqIntroBinG} with
right-hand side $f_0\in L_2(G)$, $f_i\in W^{1/2}(\Gamma_i)$ if $u$ satisfies nonlocal
conditions~\eqref{eqIntroBinG} (the equalities are understood as those in
$W^{1/2}(\Gamma_i)$) and Eq.~\eqref{eqIntroPinG} in the sense of distributions. Assume
that $f_i\in W^{3/2}(\Gamma_i)$. Then one can show   that any generalized solution of
problem~\eqref{eqIntroPinG}, \eqref{eqIntroBinG} belongs to the space $W^2$ outside of an
arbitrarily small neighborhood of the points $g$ and $h$. Clearly, the behavior of
solutions near the points $g$ and $h$ is affected by the behavior of the coefficients
$b_1$, $b_2$, and $a$ near these points. However, the influence of the coefficients $b_i$
is principally different from that of the coefficient $a$. This phenomenon is explained
by the fact that the coefficients $b_i$ (for $y$ being in a small neighborhood of the
points $g$ and $h$) correspond to nonlocal terms supported {\em near} the set $\{g,h\}$
(in the general case, such terms correspond to operators $\mathbf B_{i\mu}^1$), whereas
the coefficient $a$ corresponds to a nonlocal term supported {\em outside} of some
neighborhood of the set $\{g,h\}$ (in the general case, such terms correspond to abstract
operators $\mathbf B_{i\mu}^2$).

It was proved in~\cite{GurAdvDiffEq}   that the smoothness of generalized solutions
preserves if $b_1(g)+b_2(g)\le-2$ or $b_1(g)+b_2(g)>0$ and can be violated if
$-2<b_1(g)+b_2(g)<0$. If $b_1(g)+b_2(g)=0$, we have the ``border'' case: the smoothness
of generalized solutions depends on the fulfillment of some integral  consistency
condition imposed on the right-hand sides $ f_i $ and the coefficients $b_i$.

Now we illustrate another phenomenon arising in the border case. Assume that
$b_1(y)\equiv b_2(y)\equiv 0$. Let $a(y)=0$ in some neighborhood of the point~$h$ and
$\Omega(g)\in G$. Then the {\em support of nonlocal terms lies strictly inside the
domain~$G$.} However, if $a(g)\ne0$ or $(\partial a/\partial\tau_g)|_{y=g}\ne0$, where
$\tau_g$ denotes the unit vector tangent to $\partial G$ at the point $g$, then the
smoothness of generalized solutions of problem~\eqref{eqIntroPinG}, \eqref{eqIntroBinG}
(even with homogeneous nonlocal conditions: $\{f_i\}=0$) can be violated.

The phenomena similar to the above  occur in the  case of elliptic equations of order
$2m$ with general nonlocal conditions, which we study in the present paper. In
Sec.~\ref{sectStatement}, we provide the setting of  nonlocal problem and introduce the
notion of a generalized solution $u\in W^\ell(G)$ of the   problem for any integral
$0\le\ell\le 2m-1$.

It turns out that the smoothness of generalized solutions essentially depends on the
location of eigenvalues and the structure of root functions of some auxiliary nonlocal
operator $\tilde \cL(\lambda)$, $\lambda\in\bbC$, corresponding to the conjugation
points.

Let  $\Lambda$ denote the set of all eigenvalues of
 $\tilde{\mathcal L}(\lambda)$ lying in the strip
$1-2m<\Im\lambda<1-\ell$ (this set might be empty). In Sec.~\ref{sectNoEigen} we assume
that the line  $\Im\lambda=1-2m$ has no eigenvalues of the operator  $\tilde{\mathcal
L}(\lambda)$ and find sufficient conditions on the eigenvalues from the set $\Lambda$
under which any generalized solution of nonlocal problem belongs to $W^{2m}(G)$.

In Sec.~\ref{sectProperEigen}, we investigate the ``border'' case in which the line
$\Im\lambda=1-2m$ contains the unique eigenvalue $i(1-2m)$ of   $\tilde{\mathcal
L}(\lambda)$ and this eigenvalue is proper (see Definition~\ref{defRegEigVal}). We show
that, under the same conditions on the eigenvalues of $\tilde\cL(\lambda)$ as in
Sec.~\ref{sectNoEigen}, the smoothness of generalized solutions preserves if and only if
the right-hand side of the problem and the coefficients at the nonlocal terms satisfy
some integral consistency conditions near the conjugation points.

In Sec.~\ref{sectImproperEigen}, we show that the sufficient conditions from the previous
sections are also necessary for any generalized solution to be smooth.

Some facts concerning the functional spaces and model nonlocal problems in plane angles
which we use throughout the paper are collected in Appendix.

   The results of this paper have been obtained during the author's work
at the research group of Professor J{\"a}ger (Heidelberg University) in the framework of
the project supported by the Humboldt Foundation. The author also expresses his gratitude
to Professor Skubachevskii for attention.

\section{Setting of Nonlocal Problems in Bounded Domains}\label{sectStatement}

\subsection{Setting of the Problem}\label{subsectStatement}

Let $X$ be a domain in $\mathbb R^n$, $n=1,2$. Denote by $C_0^\infty(X)$ the set of
functions infinitely differentiable on $\overline{ X}$ and compactly supported in $X$. If
$M$ is a union of finitely many points (for $n=1,2$) or curves (for $n=2$)  lying in
$\overline X$, we denote by $C_0^\infty(\overline X\setminus M)$ the set of functions
infinitely differentiable on $\overline{ X}$ and compactly supported in $\overline
X\setminus M$.

Let $G\subset{\mathbb R}^2$ be a bounded domain with boundary $\partial G$. Consider a
set ${\mathcal K}\subset\partial G$ consisting of finitely many points. Let $\partial
G\setminus{\mathcal K}=\bigcup\limits_{i=1}^{N}\Gamma_i$, where $\Gamma_i$ are open (in
the topology of $\partial G$) $C^\infty$ curves. Assume that the domain $G$ is a plane
angle in some neighborhood of each point $g\in{\mathcal K}$.

For an integral $k\ge0$, denote by $W^k(G)=W_2^k(G)$ the Sobolev space with the norm
$$
\|u\|_{W^k(G)}=\left(\sum\limits_{|\alpha|\le k}\int_G |D^\alpha u(y)|^2\,dy\right)^{1/2}
$$
(set $W^0(G)=L_2(G)$ for $k=0$), where $\alpha=(\alpha_1,\dots,\alpha_n)$,
$|\alpha|=\alpha_1+\dots+\alpha_n$, $D^\alpha=D_1^{\alpha_1}\dots D_n^{\alpha_n}$,
$D_j=-i\partial/\partial x_j$.

 For an integral
$k\ge1$, we introduce the space $W^{k-1/2}(\Gamma)$ of traces on a smooth curve
$\Gamma\subset\overline{ G}$ with the norm
$$
\|\psi\|_{W^{k-1/2}(\Gamma)}=\inf\|u\|_{W^k(G)}\quad (u\in W^k(G):\ u|_\Gamma=\psi).
$$

Along with Sobolev spaces, we will use weighted spaces (the Kondrat'ev spaces).
 Let $Q=\{y\in{\mathbb R}^2:\ r>0,\
|\omega|<\omega_0\}$,  $Q=\{y\in{\mathbb R}^2:\ 0<r<d,\ |\omega|<\omega_0\}$,
$0<\omega_0<\pi$, $d>0$, or $Q=G$. We denote by $\mathcal M$ the set $\{0\}$ in the first
and second cases and the set $\mathcal K$ in the third case. Introduce the space
$H_a^k(Q)=H_a^k(Q,\mathcal M)$ as a completion of the set $C_0^\infty(\overline{
Q}\setminus \mathcal M)$ with respect to the norm
$$
 \|u\|_{H_a^k(Q)}=\left(
    \sum_{|\alpha|\le k}\int_Q \rho^{2(a-k+|\alpha|)} |D^\alpha u(y)|^2 dy
                                       \right)^{1/2},
$$
where $a\in \mathbb R$, $k\ge 0$ is an integral, and $\rho=\rho(y)=\dist(y,\mathcal M)$.
For an integral $k\ge1$, denote by $H_a^{k-1/2}(\Gamma)$  the set of traces on a smooth
curve $\Gamma\subset\overline{ Q}$ with the norm
\begin{equation}\label{eqTraceNormH}
\|\psi\|_{H_a^{k-1/2}(\Gamma)}=\inf\|u\|_{H_a^k(Q)} \quad (u\in H_a^k(Q):\ u|_\Gamma =
\psi).
\end{equation}

Denote by ${\bf P}(y, D_y)={\bf P}(y,D_{y_1},D_{y_2})$ and $B_{i\mu s}(y, D_y)=B_{i\mu
s}(y,D_{y_1},D_{y_2})$ differential operators of order $2m$ and $m_{i\mu}$ ($m_{i\mu}\le
2m-1$), respectively, with complex-valued $C^\infty$ coefficients ($i=1, \dots, N;$
$\mu=1, \dots, m;$ $s=0, \dots, S_i$). Here $D_y=(D_{y_1},D_{y_2})$,
$D_{y_j}=-i\partial/\partial y_j$.

We assume  that the following condition holds for the operators ${\bf P}(y, D_y)$ and
$B_{i\mu 0}(y, D_y)$ (these operators will correspond to the ``local'' elliptic problem).

\begin{condition}\label{condElliptic}
The operator ${\bf P}(y, D_y)$ is properly elliptic on $\overline{G}$,  and the   system
$\{B_{i\mu 0}(y, D_y)\}_{\mu=1}^m$ satisfies the Lopatinsky condition with respect to the
operator  ${\bf P}(y, D_y)$ for all $i=1, \dots, N$ and $y\in\overline{\Gamma_i}$.
\end{condition}

We denote
$$
\bB_{i\mu}^0u=B_{i\mu 0}(y, D_y)u,\quad y\in\Gamma_i,\ i=1,\dots,N,\ \mu=1,,\dots,m.
$$

\smallskip

For any closed set $\mathcal M$, we denote its $\varepsilon$-neighborhood by $\mathcal
O_{\varepsilon}(\mathcal M)$, i.e.,
$$
\mathcal O_{\varepsilon}(\mathcal M)=\{y\in \mathbb R^2:\ \dist(y, \mathcal
M)<\varepsilon\},\qquad {\varepsilon}>0.
$$

Now we introduce operators corresponding to nonlocal terms supported near the set
$\mathcal K$. Let $\Omega_{is}$ ($i=1, \dots, N;$ $s=1, \dots, S_i$) be $C^\infty$
diffeomorphisms taking some neighborhood ${\mathcal O}_i$ of the curve
$\overline{\Gamma_i\cap\mathcal O_{{\varepsilon}}(\mathcal K)}$ to the set
$\Omega_{is}({\mathcal O}_i)$ in such a way that $\Omega_{is}(\Gamma_i\cap\mathcal
O_{{\varepsilon}}(\mathcal K))\subset G$ and
\begin{equation}\label{eqOmega}
\Omega_{is}(g)\in\mathcal K\quad\text{for}\quad g\in\overline{\Gamma_i}\cap\mathcal K.
\end{equation}
Thus, the transformations $\Omega_{is}$ take the curves $\Gamma_i\cap\mathcal
O_{{\varepsilon}}(\mathcal K)$ strictly inside the domain $G$ and the set of their end
points $\overline{\Gamma_i}\cap\mathcal K$ to itself.

Let us specify the structure of the transformations $\Omega_{is}$ near the set $\mathcal
K$. Denote by $\Omega_{is}^{+1}$ the transformation $\Omega_{is}:{\mathcal
O}_i\to\Omega_{is}({\mathcal O}_i)$ and by $\Omega_{is}^{-1}:\Omega_{is}({\mathcal
O}_i)\to{\mathcal O}_i$ the inverse transformation. The set of points
$\Omega_{i_qs_q}^{\pm1}(\dots\Omega_{i_1s_1}^{\pm1}(g))\in{\mathcal K}$ ($1\le s_j\le
S_{i_j},\ j=1, \dots, q$) is said to be an {\em orbit} of the point $g\in{\mathcal K}$
and denoted by $\Orb(g)$. In other words, the orbit $\Orb(g)$ is formed by the points (of
the set $\mathcal K$) that can be obtained by consecutively applying the transformations
$\Omega_{i_js_j}^{\pm1}$ to the point $g$.

It is clear that either $\Orb(g)=\Orb(g')$ or $\Orb(g)\cap\Orb(g')=\varnothing$ for any
$g, g'\in{\mathcal K}$. In what follows, we assume that the set $\mathcal K$ consists of
one orbit  (the results   are easy to generalize for the case in which $\mathcal K$
consists of finitely many disjoint orbits, cf. Sec.~6 in~\cite{GurAdvDiffEq}). To
simplify the notation, we also assume that the set (orbit) $\mathcal K$ consists of $N$
points: $g_1,\dots,g_N$.

Take a sufficiently small number $\varepsilon$ (cf. Remark~2.3 in~\cite{GurAdvDiffEq})
such that there exist neighborhoods $\mathcal O_{\varepsilon_1}(g_j)$, $ \mathcal
O_{\varepsilon_1}(g_j)\supset\mathcal O_{\varepsilon}(g_j) $, satisfying the following
conditions:
\begin{enumerate}
\item The domain $G$ is a plane angle in the neighborhood $\mathcal O_{\varepsilon_1}(g_j)$;
\item
$\overline{\mathcal O_{\varepsilon_1}(g_j)}\cap\overline{\mathcal
O_{\varepsilon_1}(g_k)}=\varnothing$ for any $g_j,g_k\in\mathcal K$, $k\ne j$;
\item If $g_j\in\overline{\Gamma_i}$ and
$\Omega_{is}(g_j)=g_k,$ then ${\mathcal O}_{\varepsilon}(g_j)\subset\mathcal
 O_i$ and
 $\Omega_{is}\big({\mathcal
O}_{\varepsilon}(g_j)\big)\subset{\mathcal O}_{\varepsilon_1}(g_k).$
\end{enumerate}

For each point $g_j\in\overline{\Gamma_i}\cap\mathcal K$, we fix a transformation $Y_j:
y\mapsto y'(g_j)$ which is a composition of the shift by the vector
$-\overrightarrow{Og_j}$ and the rotation through some angle so that
$$
Y_j({\mathcal O}_{\varepsilon_1}(g_j))={\mathcal O}_{\varepsilon_1}(0),\qquad
Y_j(G\cap{\mathcal O}_{\varepsilon_1}(g_j))=K_j\cap{\mathcal O}_{\varepsilon_1}(0),
$$
$$
Y_j(\Gamma_i\cap{\mathcal O}_{\varepsilon_1}(g_j))=\gamma_{j\sigma}\cap{\mathcal
O}_{\varepsilon_1}(0)\quad (\sigma=1\ \text{or}\ 2),
$$
where
$$
K_j=\{y\in{\mathbb R}^2:\ r>0,\ |\omega|<\omega_j\},\quad \gamma_{j\sigma}=\{y\in\mathbb
R^2:\ r>0,\ \omega=(-1)^\sigma \omega_j\}.
$$
Here $(\omega,r)$ are the polar coordinates and $0<\omega_j<\pi$.

Let   the following condition hold (see Fig.~\ref{figTransform}).
\begin{condition}\label{condK1}
Let $g_j\in\overline{\Gamma_i}\cap\mathcal K$ and $\Omega_{is}(g_j)=g_k\in\mathcal K;$
then the transformation
$$
Y_k\circ\Omega_{is}\circ Y_j^{-1}:{\mathcal O}_{\varepsilon}(0)\to{\mathcal
O}_{\varepsilon_1}(0)
$$
is the composition of   rotation and   homothety.
\end{condition}
\begin{figure}[ht]
{ \hfill\epsfxsize132mm\epsfbox{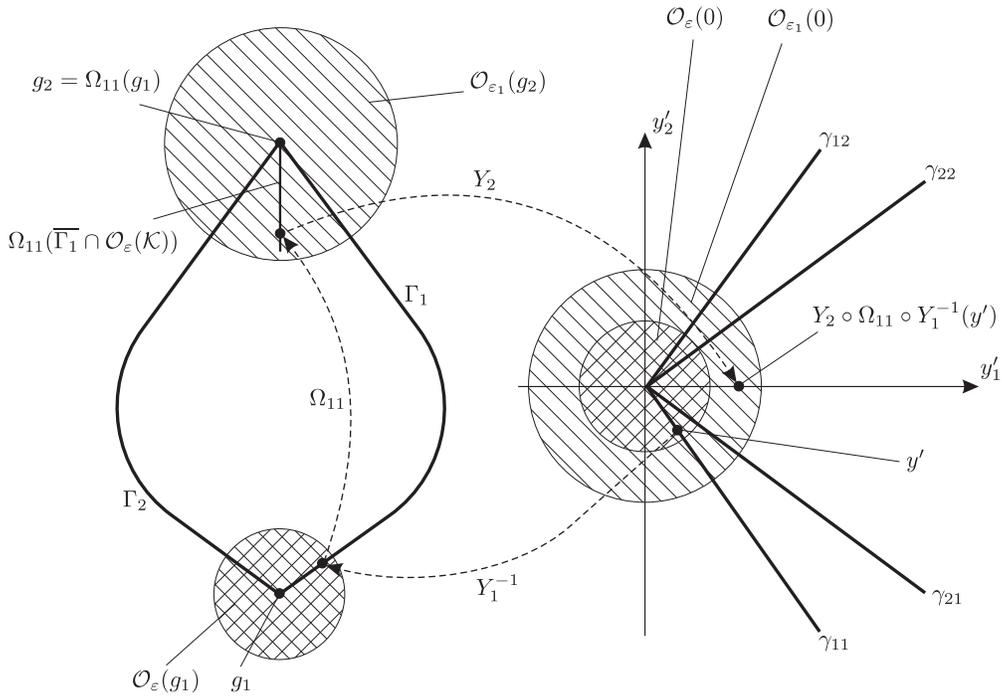}\hfill\ } \caption{The transformation $
Y_2\circ\Omega_{11}\circ Y_1^{-1}:{\mathcal O}_{\varepsilon}(0)\to{\mathcal
O}_{\varepsilon_1}(0) $ is a composition of rotation and homothety}
   \label{figTransform}
\end{figure}

\begin{remark}\label{remK1}
Condition~\ref{condK1}, together with the fact that $\Omega_{is}(\Gamma_i)\subset G$,
implies that if $g\in\Omega_{is}(\overline{\Gamma_i}\cap\mathcal
K)\cap\overline{\Gamma_j}\cap{\mathcal K}\ne\varnothing$, then the curves
$\Omega_{is}(\overline{\Gamma_i}\cap\mathcal O_\varepsilon(\mathcal K))$ and
$\overline{\Gamma_j}$ intersect at nonzero angle at the point $g$.
\end{remark}

We choose a number  $\varepsilon_0$, $0<\varepsilon_0\le\varepsilon$ possessing  the
following property: if $g_j\in\overline{\Gamma_i}$ and $\Omega_{is}(g_j)=g_k,$ then
${\mathcal O}_{\varepsilon_0}(g_k)\subset \Omega_{is} ({\mathcal O}_{\varepsilon}(g_j)
)\subset {\mathcal O}_{\varepsilon_1}(g_k)$. Consider a function $\zeta\in
C^\infty(\mathbb R^2)$ such that
$$
 \zeta(y)=1\ (y\in\mathcal O_{{\varepsilon_0}/2}(\mathcal K)),\quad
 \zeta(y)=0\ (y\notin\mathcal O_{{\varepsilon_0}}(\mathcal K)).
$$

Introduce the nonlocal operators $\mathbf B_{i\mu}^1$ by the formulas
$$
\begin{aligned}
 \mathbf B_{i\mu}^1u &=\sum\limits_{s=1}^{S_i}
   \big(B_{i\mu s}(y,
   D_y)(\zeta u)\big)\big(\Omega_{is}(y)\big), & &
   y\in\Gamma_i\cap\mathcal O_{\varepsilon}(\mathcal K), \\
 \mathbf B_{i\mu}^1u &=0, & & y\in\Gamma_i\setminus\mathcal
O_{\varepsilon}(\mathcal K),
\end{aligned}
$$
where $\big(B_{i\mu s}(y, D_y)u\big)\big(\Omega_{is}(y)\big)=B_{i\mu s}(x,
D_{x})u(x)|_{x=\Omega_{is}(y)}$. Since $\mathbf B_{i\mu}^1u=0$ for $\supp
u\subset\overline{ G}\setminus\overline{\mathcal O_{\varepsilon_0}(\mathcal K)}$, we say
that the operators $\mathbf B_{i\mu}^1$ \textit{correspond to nonlocal terms supported
near the set} $\mathcal K$.

\smallskip

Set $G_\rho=\{y\in G:  \dist(y,
\partial G)>\rho\}$ for $\rho>0$. Consider operators $\mathbf
B_{i\mu}^2$ satisfying the following condition (cf.~\cite{SkMs86,SkJMAA,GurRJMP03}).
\begin{condition}\label{condSeparK23}
There exist numbers $\varkappa_1>\varkappa_2>0$ and $\rho>0$ such that
\begin{equation}\label{eqSeparK23'}
  \|\mathbf B^2_{i\mu}u\|_{W^{2m-m_{i\mu}-1/2}(\Gamma_i)}\le c_1
  \|u\|_{W^{2m}(G\setminus\overline{\mathcal O_{\varkappa_1}(\mathcal
  K)})}\quad \forall u\in W^{2m}(G\setminus\overline{\mathcal
O_{\varkappa_1}(\mathcal
  K)}),
\end{equation}
\begin{equation}\label{eqSeparK23''}
  \|\mathbf B^2_{i\mu}u\|_{W^{2m-m_{i\mu}-1/2}
   (\Gamma_i\setminus\overline{\mathcal O_{\varkappa_2}(\mathcal K)})}\le
  c_2 \|u\|_{W^{2m}(G_\rho)} \quad \forall u\in
  W^{2m}(G_\rho),
\end{equation}
where $i=1, \dots, N$, $\mu=1, \dots, m$, and $c_1,c_2>0$ do not depend on $u$.
\end{condition}

It follows from~\eqref{eqSeparK23'} that $\mathbf B_{i\mu}^2u=0$ whenever $\supp u\subset
\mathcal O_{\varkappa_1}(\mathcal K)$. For this reason, we say that the operators
$\mathbf B_{i\mu}^2$ \textit{correspond to nonlocal terms supported outside the set}
$\mathcal K$.

\smallskip

We assume that Conditions~\ref{condElliptic}--\ref{condSeparK23} are fulfilled
throughout.

\smallskip

We study the following nonlocal elliptic boundary-value problem:
\begin{align}
 {\bf P}(y, D_y)u=f_0(y) \quad &(y\in G),\label{eqPinG}\\
     \mathbf B_{i\mu}^0 u+\mathbf B_{i\mu}^1 u+\mathbf B_{i\mu}^2 u=
   f_{i\mu}(y)\quad
    &(y\in \Gamma_i;\ i=1, \dots, N;\ \mu=1, \dots, m).\label{eqBinG}
\end{align}
Note that the points $g_j$ divide the curves on which different nonlocal conditions are
set; therefore, it is natural to say that $g_j$, $j=1,\dots,N$, are the {\em points of
conjugation of nonlocal conditions}.

Introduce the spaces of vector-valued functions
$$
\begin{aligned}
\cW^{\bm}(\pG)&=\prod\limits_{i=1}^N \prod\limits_{\mu=1}^m W^{\mimu}(\Gamma_i),\\
\cH_a^{\bm}(\pG)&=\prod\limits_{i=1}^N \prod\limits_{\mu=1}^m H_a^{\mimu}(\Gamma_i).
\end{aligned}
$$
We will always assume that $\{f_0,f_{i\mu}\}\in L_2(G)\times \cW^{\bm}(\pG)$.

From now on, we fix an integral number $\ell$ such that $0\le \ell\le2m-1$.

\begin{definition}\label{defGenSol2}
A function $u$ is called a {\em generalized solution} of problem~\eqref{eqPinG},
\eqref{eqBinG} with right-hand side $\{f_0,f_{i\mu}\}\in L_2(G)\times \mathcal
W^{\bm}(\partial G)$ if
\begin{equation}\label{eqSmoothOutsideK}
u\in W^\ell(G)\cap W^{2m}(G\setminus\overline{\cO_\delta(\cK)}) \quad \forall\delta>0
\end{equation}
 and $u$ satisfies relations~\eqref{eqPinG} a.e. and equalities  \eqref{eqBinG} in
$W^{2m-m_{i\mu}-1/2}(\Gamma_i\setminus\overline{\cO_\delta(\cK)})$ for all $ \delta>0.$
\end{definition}

Note that if $u$ satisfies~\eqref{eqSmoothOutsideK}, then $\mathbf B_{i\mu}^2 u\in
W^{2m-m_{i\mu}-1/2}(\Gamma_i)$ due to~\eqref{eqSeparK23'} and $\mathbf B_{i\mu}^1 u\in
W^{2m-m_{i\mu}-1/2}(\Gamma_i\setminus\overline{\cO_\delta(\cK)})$ for all $ \delta>0$.
Therefore, Definition~\ref{defGenSol2} does make sense.

\begin{remark}
Let $W^{-k}(G)$, $k\ge 1$,  denote the space adjoint  to $W^k(G) $ with respect to the
extension of the inner product in $L_2(G)$.

Denote  by $H_a^{-(k-1/2)}(\Gamma_i)$, $k\ge 1$, the space adjoint to
$H_{-a}^{k-1/2}(\Gamma_i)$ with respect to the extension of the inner product in
$L_2(\Gamma_i)$.

One can show that   $C^\infty(\overline{\Gamma_i})\subset H_\ell^{\ell-k+1/2}$,
$k=1,\dots,2m$. Therefore, the norm
\begin{equation}\label{eqTildeW}
\|\bu\|_{\bW^\ell(G)}=\left(\|\bu\|_{\bW^\ell(G)}^2+\sum\limits_{i=1}^N
\sum\limits_{k=1}^{2m} \left\|D_{\nu_i}^{k-1}\bu
\right\|_{H_\ell^{\ell-k+1/2}(\Gamma_i)}^2\right)^{1/2}
\end{equation}
is finite for any $\bu\in C^\infty(\oG)$, where $\nu_i$ is the outward normal to the
piece $\Gamma_i$ of the boundary and
$D_{\nu_i}^{k-1}\bu=(-i)^{k-1}\dfrac{\partial^{k-1}\bu}{\partial
\nu_i^{k-1}}\bigg|_{\Gamma_i}$. Denote by $\bW^\ell(G)$   the completion of
$C^\infty(\oG)$ in the norm~\eqref{eqTildeW}.

It follows from~\eqref{eqTildeW} that the closure $\bS$ of the mapping
$$
\bu\mapsto\{\bu|_G, D_{\nu_i}^{k-1}\bu\}\qquad (\bu\in C^\infty(\oG))
$$
establishes an isometric correspondence between $\bW^\ell(G)$ and a subspace of the
direct product
$$
W^\ell(G)\times \prod\limits_{i=1}^N\prod\limits_{k=1}^{2m}
H_\ell^{\ell-k+1/2}(\Gamma_i).
$$
We will identify $\bu\in\bW^\ell(G)$ with $\bS\bu$ and write $ \bu=\{u,u_{ik}\}\in
\bW^\ell(G). $

Then, similarly to~\cite{RMonog}, one can introduce the concept of a strong generalized
solution $\bu\in \bW^\ell(G)$ of problem~\eqref{eqPinG}, \eqref{eqBinG}. Moreover, one
can   prove that if~$\bu$ is a strong generalized solution, then the component $u \in
W^\ell(G) $ of the vector $\bu$ is a generalized solution in the sense of
Definition~\ref{defGenSol2}. Conversely, if $u  \in W^\ell(G) $ is a generalized solution
in the sense of Definition~\ref{defGenSol2},  then $\bu=\{u,D_{\nu_i}^{k-1}u\}$ belongs
to $\bW^\ell(G)$ and is a strong generalized solution. Furthermore, if the function $
\bv=\{u,v_{ik}\}\in\bW^\ell(G) $ (with the same first component $u$) is a strong
generalized solution, then $\bu=\bv$ i.e., a generalized solution uniquely determines a
strong generalized solution.
\end{remark}

\subsection{Model Problems}\label{subsectStatementNearK}

When studying problem~\eqref{eqPinG}, \eqref{eqBinG}, particular attention must be paid
to the behavior of solutions near the set ${\mathcal K}$ of conjugation points. In this
subsection, we consider corresponding model problems.

Denote by $u_j(y)$ the function $u(y)$ for $y\in{\mathcal O}_{\varepsilon_1}(g_j)$. If
$g_j\in\overline{\Gamma_i},$ $y\in{\mathcal O}_{\varepsilon}(g_j),$ and
$\Omega_{is}(y)\in{\mathcal O}_{\varepsilon_1}(g_k),$ then we denote the function
$u(\Omega_{is}(y))$ by $u_k(\Omega_{is}(y))$. In this notation, nonlocal
problem~(\ref{eqPinG}), (\ref{eqBinG}) acquires the following form in the
$\varepsilon$-neighborhood of the set (orbit) $\mathcal K$:
\begin{gather*}
 {\bf P}(y, D_y) u_j=f_0(y) \quad (y\in\mathcal O_\varepsilon(g_j)\cap
 G),\\
\begin{aligned}
B_{i\mu 0}(y, D_y)u_j(y)|_{\mathcal O_\varepsilon(g_j)\cap\Gamma_i}+
\sum\limits_{s=1}^{S_i} \big(B_{i\mu s}(y,D_y)(\zeta
u_k)\big)\big(\Omega_{is}(y)\big)\big|_{\mathcal O_\varepsilon(g_j)\cap\Gamma_i}
=\psi_{i\mu}(y) \\
\big(y\in \mathcal O_\varepsilon(g_j)\cap\Gamma_i;\ i\in\{1\le i\le N:
g_j\in\overline{\Gamma_i}\};\ j=1, \dots, N;\ \mu=1, \dots, m\big),
\end{aligned}
\end{gather*}
where
$$
\psi_{i\mu}=f_{i\mu}-\mathbf B_{i\mu}^2u.
$$

Let $y\mapsto y'(g_j)$ be the change of variables described in
Sec.~\ref{subsectStatement}. Set
$$
K_j^\varepsilon=K_j\cap\mathcal O_\varepsilon(0),\qquad
\gamma_{j\sigma}^\varepsilon=\gamma_{j\sigma}\cap\mathcal O_\varepsilon(0)
$$
and introduce the functions
\begin{equation}\label{eqytoy'}
\begin{gathered}
U_{j}(y')=u(y(y')),\quad F_{j}(y')=f_0(y(y')),\quad y'\in
K_{j}^\varepsilon,\\
F_{j\sigma\mu}(y')=f_{i\mu}(y(y')),\quad B_{j\sigma\mu}^u(y')=(\mathbf
B_{i\mu}^2u)(y(y')),\quad
y'\in\gamma_{j\sigma}^\varepsilon,\\
\Psi_{j\sigma\mu}(y')=F_{j\sigma\mu}(y')-B_{j\sigma\mu}^u(y'),\quad
y'\in\gamma_{j\sigma}^\varepsilon,
\end{gathered}
\end{equation}
where $\sigma=1$ $(\sigma=2)$ if the transformation $y\mapsto y'(g_j)$ takes $\Gamma_i$
to the side $\gamma_{j1}$ ($\gamma_{j2}$) of the angle $K_j$. Denote $y'$ by $y$ again.
Then, by virtue of Condition~\ref{condK1}, problem~\eqref{eqPinG}, \eqref{eqBinG}
acquires the form
\begin{gather}
  {\bf P}_{j}(y, D_y)U_j=F_{j}(y) \quad (y\in
  K_{j}^\varepsilon),\label{eqPinK}\\
\mathbf B_{j\sigma\mu}(y,D_y)U\equiv \sum\limits_{k,s}
       (B_{j\sigma\mu ks}(y, D_y)U_k)({\mathcal G}_{j\sigma ks}y)
    =\Psi_{j\sigma\mu}(y) \quad (y\in\gamma_{j\sigma}^\varepsilon).\label{eqBinK}
\end{gather}
Here (and below unless otherwise stated) $j, k=1, \dots, N;$ $\sigma=1, 2;$ $\mu=1,
\dots, m;$ $s=0, \dots, S_{j\sigma k}$; ${\mathbf P}_j(y, D_y)$ and $B_{j\sigma\mu ks}(y,
D_y)$ are differential operators of order $2m$ and $m_{j\sigma\mu}$ ($m_{j\sigma\mu}\le
2m-1$), respectively, with $C^\infty$ complex-valued coefficients; ${\mathcal G}_{j\sigma
ks}$ is the operator of rotation by an angle~$\omega_{j\sigma ks}$ and   homothety with a
coefficient~$\chi_{j\sigma ks}$ ($\chi_{j\sigma ks}>0$) in the $y$-plane. Moreover,
$$
|(-1)^\sigma b_{j}+\omega_{j\sigma ks}|<b_{k}\qquad\text{for}\qquad (k,s)\ne(j,0)
$$
(cf. Remark~\ref{remK1}) and
$$
\omega_{j\sigma j0}=0,\qquad \chi_{j\sigma j0}=1
$$
(i.e., ${\mathcal G}_{j\sigma j0}y\equiv y$).

Along with the operators ${\bf P}_{j}(y, D_y)$ and $\mathbf B_{j\sigma\mu}(y,D_y)$, we
consider the operators
\begin{equation}\label{eqCalPCalB}
  {\cP}_{j}(D_y),\qquad\cB_{j\sigma\mu}(D_y)U\equiv \sum\limits_{k,s}
       (B_{j\sigma\mu ks}(D_y)U_k)({\mathcal G}_{j\sigma ks}y),
\end{equation}
where ${\cP}_{j}(D_y)$ and $B_{j\sigma\mu ks}(D_y)$ are the principal homogeneous parts
of the operators $\mathbf P_j(0, D_y)$ and $B_{j\sigma\mu ks}(0, D_y)$, respectively.

We write the operators $\cP_j(D_y)$ and $B_{j\sigma\mu ks}(D_y)$ in the polar
coordinates: $ r^{-2m}\tilde{\mathcal P}_j (\omega, D_\omega, rD_r),$
$r^{-m_{j\sigma\mu}}\tilde B_{j\sigma\mu ks}(\omega, D_\omega, rD_r), $ respectively, and
consider the analytic operator-valued function\footnote{Main definitions and facts
concerning analytic operator-valued functions can be found in~\cite{GS}.}
$$
\tilde{\mathcal L}(\lambda): \prod_{j=1}^{N} W^{2m}(-\omega_j,
 \omega_j)\to \prod_{j=1}^{N}\big(L_2(-\omega_j, \omega_j) \times{\mathbb
 C}^{2m}\big),
$$
$$
 \tilde{\mathcal L}(\lambda)\varphi=\big\{\tilde{\mathcal P}_j(\omega, D_\omega, \lambda)\varphi_j,\
  {\tilde \cB}_{j\sigma\mu}(\omega, D_\omega, \lambda)\varphi\big\},
$$
where $D_\omega=-i\partial/\partial\omega$, $D_r=-i\partial/\partial r$, and
$$
{\tilde \cB}_{j\sigma\mu}(\omega, D_\omega, \lambda)\varphi= \sum\limits_{k,s}
(\chi_{j\sigma ks})^{i\lambda-m_{j\sigma\mu}}
 {\tilde B}_{j\sigma\mu ks}(\omega, D_\omega, \lambda)
              \varphi_k(\omega+\omega_{j\sigma ks})|_{\omega=(-1)^\sigma
              \omega_j}.
$$

Spectral properties of the operator $\tilde{\mathcal L}(\lambda)$ play a crucial role in
the study of smoothness of generalized solutions.  The following assertion is of
particular importance (see Lemmas~2.1 and 2.2 in \cite{SkDu90}).
\begin{lemma}\label{lSpectrum}
For any $\lambda\in\bbC$, the operator $\tilde{\mathcal L}(\lambda)$ has the Fredholm
property and $\ind \tilde{\mathcal L}(\lambda)=0$.

The spectrum of the operator $\tilde{\mathcal L}(\lambda)$ is discrete. For any numbers
$c_1<c_2$, the band $c_1<\Im\lambda<c_2$ contains at most finitely many eigenvalues of
the operator $\tilde{\mathcal L}(\lambda)$.
\end{lemma}

\section{Preservation of Smoothness of Generalized Solutions}\label{sectNoEigen}

\subsection{Formulation of the Main Result}

In this section, we study the case in which the following condition holds.
\begin{condition}\label{condNoEigen1-2m}
The line $\Im\lambda=1-2m$ contains no eigenvalues of the operator $\tilde{\mathcal
L}(\lambda)$.
\end{condition}

Let $\lambda=\lambda_0$ be an eigenvalue of the operator $\tilde{\mathcal L}(\lambda)$.

\begin{definition}[cf.~\cite{KondrTMMO67,GurRJMP03}]\label{defRegEigVal}
We say that $\lambda_0$ is a {\em proper eigenvalue} if none of the corresponding
eigenvectors $\varphi(\omega)=(\varphi_{1}(\omega),\dots, \varphi_{N}(\omega))$ has an
associated vector, while the functions $r^{i\lambda_0}\varphi_{j}(\omega)$, $j=1, \dots,
N$, are homogeneous polynomials in $y_1, y_2$ (of degree $i\lambda_0\in\bbN\cup\{0\}$).
An eigenvalue which is not proper is said to be {\em improper}.
\end{definition}

 Let $\Lambda$ be the set of all eigenvalues
of $\tilde{\mathcal L}(\lambda)$ in the band $1-2m<\Im\lambda<1-\ell$ (this set can be
empty). We also denote $i\Lambda=\{i\lambda: \lambda\in\Lambda\}$.

\begin{condition}\label{condNoEigenProperLambda}
All the eigenvalues from the set $\Lambda$ are proper.
\end{condition}

In particular, Condition~\ref{condNoEigenProperLambda} implies that $\Lambda=\varnothing$
if $\ell=2m-1$ (e.g., if $\ell=m=1$, cf.~\cite{GurAdvDiffEq}) and
$i\Lambda\subset\{\ell,\dots,2m-2\}$ if $\ell\le 2m-2$.

In the case where $\ell\le 2m-2$, we will need some additional conditions.

Let $W^{-2m}(-\omega_j,\omega_j)$ be the space adjoint to $ W^{2m}(-\omega_j,\omega_j)$.
Consider the operator $
 (\tilde{\cal L}(\lambda))^*: \prod_{j=1}^{N}\big(L_2(-\omega_j, \omega_j) \times{\mathbb
 C}^{2m}\big)\to \prod_{j=1}^{N} W^{-2m}(-\omega_j,
 \omega_j)$  which is adjoint to the operator $\tilde{\cal L}(\lambda)$.

For any $s\in\{\ell,\dots,2m-2\}$, we denote by $J_s$ the set of all indices
$(j',\sigma',\mu')$ such that
\begin{equation}\label{eqSlessM}
s\le m_{j'\sigma' \mu'}-1.
\end{equation}
We also denote by $C_s$ the space of numerical vectors $\{c_{j\sigma\mu}\}$ with complex
entries such that
$$
c_{j'\sigma'\mu'}=0,\quad (j',\sigma',\mu')\in J_s.
$$

\begin{condition}\label{condEigenMonom}
If $\ell\le 2m-2$, then the following assertions hold for any $s\in i\Lambda${\rm :}
\begin{enumerate}
\item $J_s\ne\varnothing$.
\item $\langle\{0,c_{j\sigma\mu}\},\psi\rangle=0$ for all $\{c_{j\sigma\mu}\}\in
C_s$ and $\psi\in\ker(\tilde{\mathcal L}(-is))^*$.
\item Let $\varphi_c\in \prod\limits_j W^{2m}(-\omega_j,\omega_j)$ denote a
solution of the equation $\tilde{\mathcal L}(-is)\varphi_c=\{0,c_{j\sigma\mu}\}$, where
$\{c_{j\sigma\mu}\}\in C_s$ {\rm (}this solution exists due to item~{\rm 2} and is
defined up to an arbitrary element $\varphi_0\in\ker\tilde{\mathcal L}(-is)${\rm )}. Then
$r^s\varphi_c(\omega)$ is a homogeneous polynomial {\rm (}of degree $s${\rm )} for any
$\{c_{j\sigma\mu}\}\in C_s$.
\end{enumerate}
\end{condition}

\begin{remark}\label{remCondEigenMonom}
\begin{enumerate}
\item
Part~1 in Condition~\ref{condEigenMonom} is necessary for the fulfillment of part~2. This
follows from Lemma~\ref{lSpectrum}.
\item
Part~2 is necessary and sufficient for the existence of  solutions $\varphi_c$ for all
$\{c_{j\sigma\mu}\}$ in part~3.
\end{enumerate}
\end{remark}

\begin{condition}\label{condNoEigenMonom}
If $\ell\le 2m-2$, then the following assertion holds for any
$s\in\{\ell,\dots,2m-2\}\setminus i\Lambda$.    Let $\varphi_c\in \prod_j
W^{2m}(-\omega_j,\omega_j)$ denote a solution\footnote{This solution exists and is unique
because $-is$ is not an eigenvalue of $\tilde{\mathcal L}(\lambda)$.} of the equation
$\tilde{\mathcal L}(-is)\varphi_c=\{0,c_{j\sigma\mu}\}$, where $\{c_{j\sigma\mu}\}\in
C_s$.  Then $r^s\varphi_c(\omega)$ is a homogeneous polynomial {\rm (}of degree $s${\rm
)} for any $\{c_{j\sigma\mu}\}\in C_s$.
\end{condition}

\begin{remark}\label{remMonomialSolution}
Suppose that Condition~\ref{condNoEigenProperLambda} is fulfilled.
\begin{enumerate}
\item
If Conditions~\ref{condEigenMonom} and~\ref{condNoEigenMonom} hold, then the problem
\begin{equation}\label{eqUinW^2_4'}
\cP_j(D_y)V=0,\qquad \cB_{j\sigma\mu}(D_y)V=c_{j\sigma\mu}r^{s-m_{j\sigma\mu}}
\end{equation}
admits a solution  $V(y)$ which is a homogeneous polynomial of degree $s$, provided that
$\{c_{j\sigma\mu}\}\in C_s$, where $s=\ell,\dots,2m-2$. Indeed, substituting a function
$V=r^s\varphi_c(\omega)$ into~\eqref{eqUinW^2_4'}, we obtain the equation
$\tilde{\mathcal L}(-is)\varphi_s=\{0,c_{j\sigma\mu}\}.$ Due to
Conditions~\ref{condEigenMonom} and~\ref{condNoEigenMonom}, this equation admits a
solution $\varphi_c$ such that the function $V=r^s\varphi_c(\omega)$ is a homogeneous
polynomial of degree $s$.
\item
If Condition~\ref{condEigenMonom} or~\ref{condNoEigenMonom} fails, then there is a vector
$\{c_{j\sigma\mu}\}\in C_s$ such that problem~\eqref{eqUinW^2_4'} admits a solution
\begin{equation}\label{eqUinW^2_4''}
V=r^s\varphi_c(\omega)+r^s(i\ln r)\sum\limits_{n=1}^J c_n\varphi^{(n)}(\omega),
\end{equation}
where $s\in\{\ell,\dots,2m-2\}$, $c_n\in\bbC$, $\varphi_c,\varphi^{(n)}\in
\prod_{j=1}^{N} W^{2m}(-\omega_j,
 \omega_j)$, and $J=J(s)$. Moreover, the function $V$ is not a polynomial in
 $y_1,y_2$.

Indeed, if Condition~\ref{condNoEigenMonom} fails, then the assertion is evident (with
$c_1=\dots=c_J=0$). Assume that Condition~\ref{condEigenMonom} fails. If parts 1 and 2 of
Condition~\ref{condNoEigenMonom} hold while part 3 fails, then the assertion is evident
again (with $c_1=\dots=c_J=0$). Let part 1 or 2 fail. In both cases, part 2 does not hold
(see Remark~\ref{remCondEigenMonom}). This means that there exists a proper eigenvalue
$\lambda_s=-is\in\Lambda$ and a numerical vector $\{c_{j\sigma\mu}\}\in C_s$ such that
$\{0,c_{j\sigma\mu}\}$ is not orthogonal to $\ker(\tilde{\mathcal L}(\lambda_s))^*$.

Let $\varphi^{(1)},\dots,\varphi^{(J)}$ ($J\ge1$) denote some basis in $\ker
\tilde{\mathcal L}(\lambda_s)$. Since $\lambda_s$ is a proper eigenvalue, none of the
eigenvectors $\varphi^{(n)}$ has an associate vector. We substitute a function $V$ given
by~\eqref{eqUinW^2_4''} in Eqs.~\eqref{eqUinW^2_4'}. Then we obtain
\begin{equation}\label{eqUinW^2_4'''}
\tilde \cL(\lambda_s)\varphi_c=\{0,c_{j\sigma\mu}\}-\sum\limits_{n=1}^J c_n\dfrac{d\tilde
\cL(\lambda)}{d\lambda}\bigg|_{\lambda=\lambda_s}\varphi^{(n)}.
\end{equation}
Note that $\dim\ker(\tilde{\mathcal L}(\lambda_s))^*=\dim\ker\tilde{\mathcal
L}(\lambda_s)=J$ due to Lemma~\ref{lSpectrum}. Let $\psi^{(1)},\dots,\psi^{(J)}$ denote a
basis in $\ker(\tilde{\mathcal L}(\lambda_s))^*$. By Lemma~3.2 in~\cite{GurPetr03}, the
matrix
$$
\left\|\left\langle\dfrac{d\tilde
\cL(\lambda)}{d\lambda}\bigg|_{\lambda=\lambda_s}\varphi^{(n)},\,\psi^{(k)}\right\rangle\right\|_{n,k=1,\dots,J}
$$
is nondegenerate. Therefore, we can choose the constants $c_n$ in such a way that the
right-hand side in~\eqref{eqUinW^2_4'''} is orthogonal to $\ker(\tilde{\mathcal
L}(\lambda_s))^*$; hence, there is a solution $\varphi_c$ for Eq.~\eqref{eqUinW^2_4'''}.
Moreover, since $\{0,c_{j\sigma\mu}\}$ is not orthogonal to $\ker(\tilde{\mathcal
L}(\lambda_s))^*$, it follows that the vector $(c_1,\dots,c_J)$ is nontrivial. Thus, the
function $V$ given by~\eqref{eqUinW^2_4''} is not a polynomial in $y_1,y_2$.
\end{enumerate}
\end{remark}

The main result of this section is as follows.

\begin{theorem}\label{thuinW^2NoEigen}
Let Conditions~$\ref{condNoEigen1-2m}$--$\ref{condNoEigenMonom}$ hold and $u$ be a
generalized solution of problem~\eqref{eqPinG}, \eqref{eqBinG} with right-hand side
$\{f_0,f_{i\mu}\}\in L_2(G)\times \mathcal W^{\bm}(\partial G)$. Then $u\in W^{2m}(G)$.
\end{theorem}

\subsection{Proof of the Main Result}\label{subsecProofMainResult}

Let $U_j(y')=u_j(y(y'))$, $j=1,\dots,N$, be the functions corresponding to the set
(orbit) $\mathcal K$ and satisfying problem~\eqref{eqPinK}, \eqref{eqBinK} with
right-hand side $\{F_j, \Psi_{j\sigma\mu}\}$ (see Sec.~\ref{subsectStatementNearK}).

Set
\begin{equation}\label{eqd1d2}
D_\chi=2\max\{\chi_{j\sigma ks}\},\qquad d_\chi=\min\{\chi_{j\sigma ks}\}/2.
\end{equation}
Let $\varepsilon>0$ be so small that $D_\chi \varepsilon<\varepsilon_1$ (where
$\varepsilon$ and $\varepsilon_1$ are defined in Sec.~\ref{subsectStatement}).

Introduce the spaces of vector-valued functions
\begin{equation}\label{eqSpacesCal1}
\mathcal W^k(K^\varepsilon)=\prod\limits_j W^k(K_j^\varepsilon),\quad \mathcal \mathcal
\mathcal H_a^k(K^\varepsilon)=\prod\limits_j H_a^k(K_j^\varepsilon), \quad k\ge0;
\end{equation}
\begin{equation}\label{eqSpacesCal2}
\begin{aligned}
\mathcal W^{\bm}(\gamma^\varepsilon)&=
\prod\limits_{j,\sigma}W^{\mjsigma}(\gamma_{j\sigma}^\varepsilon),\\
\mathcal H_a^{\bm}(\gamma^\varepsilon)&=
\prod\limits_{j,\sigma}H_a^{\mjsigma}(\gamma_{j\sigma}^\varepsilon).
\end{aligned}
\end{equation}
Similarly, one can introduce the spaces $\mathcal W^k(K)$, $\mathcal H_a^k(K)$, $\mathcal
W^{\bm}(\gamma)$, and $\mathcal H_a^{\bm}(\gamma)$.

Since any generalized solution
 $
u\in W^{2m}\bigl(G\setminus\overline{\mathcal O_\delta(\mathcal K)}\bigr)$ for  any
$\delta>0$ by definition, it follows that
\begin{equation}\label{eqU_jW2loc}
U_j\in W^{2m}(K_{j}^{\varepsilon_1}\setminus\overline{\cO_\delta(0)})\qquad \forall
\delta>0.
\end{equation}

It follows from the belonging $U\in \mathcal W^\ell(K^{\varepsilon_1})$ that
\begin{equation}\label{eqUa-10}
U\in   \mathcal H_{0}^0(K^{\varepsilon_1}).
\end{equation}
Further, we have (see~\eqref{eqPinK}, \eqref{eqBinK}) $\{F_j\}\in\mathcal
W^0(K^{\varepsilon})$ and, by the belonging $f_{i\mu}\in W^{\mimu}(\Gamma_i)$, by
relation~\eqref{eqSmoothOutsideK}, and by estimate~\eqref{eqSeparK23'}, we have
$\{\Psi_{j\sigma\mu}\}\in \mathcal W^{\bm}(\gamma^\varepsilon)$. Therefore,
\begin{equation}\label{eqf1+a3/2}
\{F_j\}\in \mathcal H_{2m}^{0}(K^{\varepsilon}),\quad \{\Psi_{j\sigma\mu}\}\in \mathcal
H_{2m}^{\bm}(\gamma^\varepsilon).
\end{equation}

It follows from relations~\eqref{eqU_jW2loc}--\eqref{eqf1+a3/2} and from
Lemma~\ref{lAppL2.3GurMatZam05} that
\begin{equation}\label{eqU_jH1+a}
U\in\mathcal  H_{2m}^{2m}(K^{\varepsilon_1}).
\end{equation}
To prove Theorem~\ref{thuinW^2NoEigen}, it suffices to show that $U\in \mathcal
W^{2m}(K^{\varepsilon})$.

\begin{lemma}\label{lU=C+m2m}
Let $U\in\mathcal W^\ell(K^\varepsilon)$, $U_j$ satisfy relations~\eqref{eqU_jW2loc}, and
$U$  be a solution\footnote{Since $U\in\cH_{2m}^{2m}(K^{\varepsilon_1})$ due
to~\eqref{eqU_jH1+a} and
$\{F_j,\Psi_{j\sigma\mu}\}\in\cH_{2m}^0(K^\varepsilon)\times\cH_{2m}^\bm(\gamma^\varepsilon)$,
 relations~\eqref{eqPinK}, \eqref{eqBinK} can be understood as
equalities in the corresponding weighted spaces.} of problem~\eqref{eqPinK},
\eqref{eqBinK} with right-hand side $\{F_j, \Psi_{j\sigma\mu}\}\in \mathcal
W^0(K^{\varepsilon})\times\mathcal W^{\bm}(\gamma^\varepsilon)$. Then
\begin{equation}\label{eqU=C+}
 U=Q+\hat U,
\end{equation}
where $\hat U\in\mathcal  H_{2m-\ell}^{2m}(K^\varepsilon)$  and $Q=(Q_1, \dots, Q_N)$ is
a polynomial vector of degree\footnote{Saying ``a polynomial of degree $s$,'' we always
mean ``a polynomial of degree no greater than~$s$.'' We mean that the polynomial equals
zero if $s<0$. } $\ell-1$.
\end{lemma}
\begin{proof} 1. Due to~\eqref{eqU_jH1+a}, it suffices to consider the case $\ell\ge1$. Let
$\delta$ be an arbitrary number such that $0<\delta<1$. By Lemma~4.11
in~\cite{KondrTMMO67}, for each function $\Psi_{j\sigma\mu}\in\mathcal
W^{\mjsigma}(\gamma_{j\sigma}^\varepsilon)$, there is a polynomial $P_{j\sigma\mu}(r)$ of
degree $2m-m_{j\sigma\mu}-2$  such that
$$
\{\Psi_{j\sigma\mu}-P_{j\sigma\mu}\}\in \cH_{2m-\ell-\delta}^{\bm}(\gamma^\varepsilon).
$$
  Using
Lemma~\ref{lAppL4.3GurPetr03}, one can construct a function
\begin{equation}\label{eqW1ExplicitForm}
W^1=\sum\limits_{s=0}^{\ell-1}\sum\limits_{l=0}^{l_1} r^s(i\ln
r)^l\varphi_{sl}^1(\omega)\in \mathcal H_{2m}^{2m}(K^{\varepsilon}),
\end{equation}
where   $\varphi_{sl}^1\in \prod\limits_j W^{2m}(-\omega_j,\omega_j)$, such that
$$
\{\mathbf P_j(y,D_y)W^1_j\}\in \cH_{2m-\ell-\delta}^0(K^\varepsilon),\  \{\mathbf
B_{j\sigma\mu}(y,D_y)W^1-P_{j\sigma\mu}\}\in
\cH_{2m-\ell-\delta}^{\bm}(\gamma^\varepsilon).
$$

Therefore, $ \{\mathbf P_j(y,D_y)(U_j-W^1_j)\}\in \cH_{2m-\ell-\delta}^0(K^\varepsilon),$
$\{\mathbf B_{j\sigma\mu}(y,D_y)(U_j-W^1)\}\in
\cH_{2m-\ell-\delta}^\bm(\gamma^\varepsilon). $

It follows from~\eqref{eqU_jH1+a} and~\eqref{eqW1ExplicitForm} that $U-W^1\in\mathcal
H_{2m}^{2m}(K^{\varepsilon})$. Due to Lemma~\ref{lSpectrum}, we can choose a number
$\delta$, $0<\delta<1$, in such a way that the band $1-\ell-\delta\le\Im\lambda<1-\ell$
has no eigenvalues of $\tilde{\mathcal L}(\lambda)$. Therefore, applying
Lemma~\ref{lAppTh2.2GurPetr03} and Lemma~\ref{lAppL4.3GurPetr03}, we obtain
$$
U-W^1=W^{2}+\hat U,
$$
where
$$
W^{2}=\sum\limits_{n=1}^{n_0}\sum\limits_{l=0}^{l_2} r^{i\mu_n}(i\ln
r)^l\varphi_{nl}^2(\omega),
$$
$\{\mu_1,\dots,\mu_{n_0}\}$ is the set of all eigenvalues lying in the band
$1-\ell\le\Im\lambda<1$ (in fact, we have to consider the eigenvalues in the band
$1-\ell-\delta\le\Im\lambda<1$, but the band $1-\ell-\delta\le\Im\lambda<1-\ell$ has no
eigenvalues by the choice of $\delta$), $\varphi_{nl}^2\in \prod\limits_j
W^{2m}(-\omega_j,\omega_j)$, and $\hat U\in \mathcal
H_{2m-\ell-\delta}^{2m}(K^{\varepsilon})\subset\mathcal
H_{2m-\ell}^{2m}(K^{\varepsilon})$.

Since $s\le \ell-1$ (in the formula for $W^1$), $\Re i\mu_n\le \ell-1$ (in the formula
for $W^{2}$), and $W^1+W^{2}=U-\hat U\in \mathcal W^\ell(K^{\varepsilon})$, it follows
from Lemma~\ref{lAppL4.20Kondr} that $W^1+W^{2}$ is a polynomial vector of degree
$\ell-1$.
\end{proof}

\begin{lemma}\label{lU=C+}
Let the hypotheses of Lemma~$\ref{lU=C+m2m}$ be fulfilled, and let
Conditions~$\ref{condNoEigenProperLambda}$--$\ref{condNoEigenMonom}$ hold. Then
\begin{equation}\label{eqlU=C+0}
 U=W+U'
\end{equation}
 where $W=(W_1, \dots, W_N)$ is a polynomial vector of degree
$2m-2$,  $U'\in\mathcal  H_{\delta}^{2m}(K^\varepsilon)$ {\rm (}$\delta$ is such that
$0<\delta<1$ and the band $1-2m<\Im\lambda\le 1-2m+\delta$ contains no eigenvalues of
$\tilde\cL(\lambda)${\rm )}, and
\begin{equation}\label{eqlU=C+00}
\begin{aligned}
\{\bP_j(y,D_y)U_j'\}&\in\cH_0^0(K^\varepsilon),\\
\{\bB_{j\sigma\mu}(y,D_y)U'\}&\in\cH_\delta^\bm(\gamma^\varepsilon)\cap
\cW^\bm(\gamma^\varepsilon).
\end{aligned}
\end{equation}
\end{lemma}
\begin{proof} 1. Consider the function $\hat U$ defined by Lemma~\ref{lU=C+m2m}. The function $\hat
U$ belongs to $\cH_{2m-\ell}^{2m}(K^\varepsilon)$, and, by virtue of
relations~\eqref{eqPinK}, \eqref{eqBinK}, and \eqref{eqU=C+}, it is a solution of the
problem
\begin{equation}\label{eqUinW^2_0}
\begin{aligned}
 {\mathbf P}_{j}(y,D_y)\hat U_j&=F_j-{\mathbf P}_{j}(y,D_y)Q_j
& &(y\in K_j^\varepsilon),\\
 {\mathbf
B}_{j\sigma\mu}(y,D_y)\hat U &=\Psi_{j\sigma\mu}-{\mathbf B}_{j\sigma\mu}(y,D_y)Q & &
(y\in\gamma_{j\sigma}^\varepsilon).
\end{aligned}
\end{equation}
Since $\{F_j\}\in\mathcal W^0(K^\varepsilon)$ and $Q$ is a polynomial vector, it follows
that
\begin{equation}\label{eqUinW^2_1}
\{F_j-{\mathbf P}_{j}(y,D_y)Q_j\}\in\mathcal H_0^0(K^\varepsilon).
\end{equation}
Further, $\Psi_{j\sigma\mu}-{\mathbf B}_{j\sigma\mu}(y,D_y)Q\in
W^\mjsigma(\gamma_j^\varepsilon)$. Hence, by Lemma~4.11 in~\cite{KondrTMMO67}, there
exists a polynomial $P_{j\sigma\mu}(r)$ of degree $2m-m_{j\sigma\mu}-2$   such that
\begin{equation}\label{eqUinW^2_2}
\{\Psi_{j\sigma\mu}-{\mathbf B}_{j\sigma\mu}(y,D_y)Q-P_{j\sigma\mu}\}\in\mathcal
H_\delta^{\bm}(\gamma^\varepsilon)\cap \mathcal W^{\bm}(\gamma^\varepsilon)
\end{equation}
for any $ 0<\delta<1$. Moreover, since
$$
\{\Psi_{j\sigma\mu}-{\mathbf B}_{j\sigma\mu}(y,D_y)Q\}=\{\bB_{j\sigma\mu}(y,D_y)\hat
U\}\in\cH_{2m-\ell}^{\bm}(\gamma^\varepsilon),
$$
we see that each polynomial $P_{j\sigma\mu}(r)$ consists of monomials of degree
$\max(0,\ell-m_{j\sigma\mu}),\dots,2m-m_{j\sigma\mu}-2$ (the polynomial
$P_{j\sigma\mu}(r)$ is absent if $\ell=2m-1$).

2. We write each polynomial $P_{j\sigma\mu}(r)$ as follows:
\begin{equation}\label{eqUinW^2_3}
P_{j\sigma\mu}(r)=c_{j\sigma\mu}r^{\ell-m_{j\sigma\mu}}+c_{j\sigma\mu}'r^{\ell-m_{j\sigma\mu}+1}+\dots,
\end{equation}
where, in particular, $c_{j\sigma\mu}=0$ for all $j,\sigma,\mu$ such that $\ell\le
m_{j\sigma\mu}-1$ (cf.~\eqref{eqSlessM} for $s=\ell$). Therefore, $\{c_{j\sigma\mu}\}\in
C_\ell$.

We consider the auxiliary problem
\begin{equation}\label{eqUinW^2_4}
\cP_j(D_y)W^\ell=0,\qquad
\cB_{j\sigma\mu}(D_y)W^\ell=c_{j\sigma\mu}r^{\ell-m_{j\sigma\mu}},
\end{equation}
where $\cP_j(D_y)$ and $\cB_{j\sigma\mu}(D_y)$ are the same as in~\eqref{eqCalPCalB}. By
virtue of Conditions~\ref{condEigenMonom} and~\ref{condNoEigenMonom} (see
Remark~\ref{remMonomialSolution}), there exists a solution $W^\ell(y)$ of
problem~\eqref{eqUinW^2_4} such that $W^\ell(y)$ is a homogeneous polynomial of degree
$\ell$.

Using~\eqref{eqUinW^2_3} and~\eqref{eqUinW^2_4} and expanding the coefficients of
$\bB_{j\sigma\mu}(y,D_y)$ by the Taylor formula, we obtain
\begin{equation}\label{eqUinW^2_5}
\begin{aligned}
\{\bP_j(y,D_y)W_j^\ell\}&\in\cH_0^0(K^\varepsilon),\\
\{\bB_{j\sigma\mu}(y,D_y)W^\ell-P_{j\sigma\mu}+P'_{j\sigma\mu}\}&\in\cH_\delta^\bm(\gamma^\varepsilon)\cap
\cW^\bm(\gamma^\varepsilon),
\end{aligned}
\end{equation}
where $P'_{j\sigma\mu}(r)$ is a polynomial consisting of monomials of degree
$\max(0,\ell-m_{j\sigma\mu}+1),\dots, 2m-m_{j\sigma\mu}-2$.

It follows from~\eqref{eqUinW^2_1}, \eqref{eqUinW^2_2}, and~\eqref{eqUinW^2_5} that
\begin{equation}\label{eqUinW^2_6}
\begin{aligned}
\{F_j-\bP_j(y,D_y)(Q_j+W_j^\ell)\}&\in\cH_0^0(K^\varepsilon),\\
\{\Psi_{j\sigma\mu}-\bB_{j\sigma\mu}(y,D_y)(Q+W^\ell)-P'_{j\sigma\mu}\}&\in\cH_\delta^\bm(\gamma^\varepsilon)\cap
\cW^\bm(\gamma^\varepsilon).
\end{aligned}
\end{equation}

3. Repeating the procedure described in item~2 finitely many times (and using
Conditions~\ref{condEigenMonom} and~\ref{condNoEigenMonom} each time), we obtain
\begin{equation}\label{eqUinW^2_7}
\begin{aligned}
&\{F_j-\bP_j(y,D_y)(Q_j+W_j^\ell+\dots+W_j^{2m-2})\} \in\cH_0^0(K^\varepsilon),\\
&\{\Psi_{j\sigma\mu}-\bB_{j\sigma\mu}(y,D_y)(Q+W^\ell+\dots W^{2m-2}
)\}\\
&\qquad\qquad\qquad\qquad\in\cH_\delta^\bm(\gamma^\varepsilon)\cap
\cW^\bm(\gamma^\varepsilon),
\end{aligned}
\end{equation}
where $W^s$ is a homogeneous polynomial vector of degree $s$, $s=\ell,\dots,2m-2$ (note
that a homogeneous polynomial vector of degree $2m-1$ already belongs to
$\cH_\delta^{2m}(K^\varepsilon)$). If $\ell=2m-1$, then the polynomials $W^s$
in~\eqref{eqUinW^2_7} are absent; in this case, the second relation in~\eqref{eqUinW^2_7}
follows from~\eqref{eqUinW^2_2}, where $P_{j\sigma\mu}$ is absent.

Combining~\eqref{eqUinW^2_0} and~\eqref{eqUinW^2_7} yields
\begin{equation}\label{eqUinW^2_8}
\begin{aligned}
\{\bP_j(y,D_y)(\hat U_j-W_j^\ell-\dots-W_j^{2m-2})\}&\in\cH_0^0(K^\varepsilon),\\
\{\bB_{j\sigma\mu}(y,D_y)(\hat U-W^\ell-\dots - W^{2m-2}
)\}&\in\cH_\delta^\bm(\gamma^\varepsilon)\cap \cW^\bm(\gamma^\varepsilon).
\end{aligned}
\end{equation}

4. Since the line $\Im\lambda=1-2m+\delta$ has no eigenvalues of $\tilde{\mathcal
L}(\lambda)$ and relations~\eqref{eqUinW^2_8} hold, it follows from
Lemma~\ref{lAppTh2.2GurPetr03}, Lemma~\ref{lAppL4.3GurPetr03}, and
Conditions~$\ref{condNoEigenProperLambda}$--$\ref{condNoEigenMonom}$ that the function
$\hat U+W^\ell+\dots+ W^{2m-2}$ belongs to the space $\cH_\delta^{2m}(K^\varepsilon)$ up
to a polynomial consisting of monomials of degree $\min\limits_{s\in i\Lambda}s,\dots,
2m-2$ (this polynomial is absent if $\ell=2m-1$). In other words, there is a polynomial
vector $\hat W$ consisting of monomials of degree $l,\dots, 2m-2$ such that
\begin{equation}\label{eqUinW^2_9}
\begin{aligned}
\hat U+\hat W&\in\cH_\delta^{2m}(K^\varepsilon)\\
\{\bP_j(y,D_y)(\hat U_j+\hat W_j)\}&\in\cH_0^0(K^\varepsilon),\\
\{\bB_{j\sigma\mu}(y,D_y)(\hat U+\hat W)\}&\in\cH_\delta^\bm(\gamma^\varepsilon)\cap
\cW^\bm(\gamma^\varepsilon).
\end{aligned}
\end{equation}
Now the conclusion of the lemma follows from Lemma~\ref{lU=C+m2m} and from
relations~\eqref{eqUinW^2_9}
\end{proof}

\begin{lemma}\label{lUinW^2}
Let the hypotheses of Lemma~$\ref{lU=C+m2m}$ be fulfilled, and let
Conditions~$\ref{condNoEigen1-2m}$--$\ref{condNoEigenMonom}$ hold. Then    $U \in\mathcal
W^{2m}(K^\varepsilon)$.
\end{lemma}
\begin{proof} It follows from~\eqref{eqlU=C+00} and from Lemma~\ref{lAppL2.4GurRJMP03} that there
exists a function $ V\in\mathcal H_{\delta}^{2m}(K)\cap\mathcal W^{2m}(K) $ such that
\begin{equation}\label{eqUinW^2_10}
\begin{aligned}
\{\bP_j(y,D_y)(U_j'-V_j)\}&\in\cH_0^0(K^\varepsilon),\\
\{\bB_{j\sigma\mu}(y,D_y)(U'-V )\}&\in\cH_0^\bm(\gamma^\varepsilon).
\end{aligned}
\end{equation}

Due to~\eqref{eqUinW^2_10} and the fact that the strip $1-2m\le \Im\lambda\le
1-2m+\delta$ contains no eigenvalues of $\tilde{\mathcal L}(\lambda)$, we can use
Lemma~\ref{lAppTh2.2GurPetr03}  to obtain that $ U'-V\in\cH_0^{2m}(K^\varepsilon)\subset
\cW^{2m}(K^\varepsilon)$. Combining this relation with Lemma~\ref{lU=C+} completes the
proof.
\end{proof}

Theorem~\ref{thuinW^2NoEigen} results from~\eqref{eqSmoothOutsideK} and from
Lemma~\ref{lUinW^2}.

\section{The Border Case: Consistency Conditions}\label{sectProperEigen}

\subsection{Behavior of Generalized Solutions near the Conjugation Points}\label{subsectuFixed}

Let $\Lambda$ be the same set of eigenvalues of $\tilde\cL(\lambda)$ as in
Sec.~\ref{sectNoEigen}. In this section, we consider the following condition instead of
Condition~\ref{condNoEigen1-2m}.

\begin{condition}\label{condProperEigen}
The line $\Im\lambda=1-2m$ contains only the  eigenvalue $\lambda=i(1-2m)$ of the
operator $\tilde{\mathcal L}(\lambda)$. This eigenvalue is a proper one.
\end{condition}

The principal difference between the results of this section and those of
Sec.~\ref{sectNoEigen} is related to the behavior of generalized solutions near the set
(orbit) $\mathcal K$. If Condition~\ref{condProperEigen} holds, then Lemma~\ref{lU=C+}
remains valid. However, the conclusion of Lemma~\ref{lUinW^2} is no longer true because
Lemma~\ref{lAppL2.4GurRJMP03}   is inapplicable when the line $\Im\lambda=1-2m$ contains
an   eigenvalue of $\tilde{\mathcal L}(\lambda)$. In this section, we make use of other
results from~\cite{GurRJMP03}. To do this, we impose certain consistency conditions on
the behavior of the functions $f_{i\mu}$ and   the coefficients of nonlocal terms near
the set (orbit) $\mathcal K$.

\smallskip

Let $\tau_{j\sigma}$ be the unit vector co-directed with the ray~$\gamma_{j\sigma}$.
Consider the operators
$$
 \frac{\partial^{2m-m_{j\sigma\mu}-1}} {\partial
  \tau_{j\sigma}^{2m-m_{j\sigma\mu}-1}}{\mathcal B}_{j\sigma\mu}U\equiv
  \frac{\partial^{2m-m_{j\sigma\mu}-1}} {\partial
  \tau_{j\sigma}^{2m-m_{j\sigma\mu}-1}} \left(\sum\limits_{k,s}(B_{j\sigma\mu
  ks}(D_y)U_k)({\mathcal G}_{j\sigma ks}y)\right).
$$
Using the chain rule, we can write
\begin{equation}\label{eqDiffB}
\frac{\partial^{2m-m_{j\sigma\mu}-1}} {\partial
  \tau_{j\sigma}^{2m-m_{j\sigma\mu}-1}}{\mathcal B}_{j\sigma\mu}U\equiv
  \sum\limits_{k,s}(\hat B_{j\sigma\mu
  ks}(D_y)U_k)({\mathcal G}_{j\sigma ks}y).
\end{equation}
where $\hat B_{j\sigma\mu ks}(D_y)$ are some homogeneous differential operators of order
$2m-1$ with constant coefficients. Formally replacing the nonlocal operators by the
corresponding local operators in~\eqref{eqDiffB}, we introduce the operators
\begin{equation}\label{eqSystemB}
 \hat{\mathcal B}_{j\sigma\mu}(D_y)U\equiv
 \sum\limits_{k,s}\hat B_{j\sigma\mu ks}(D_y)U_k(y).
\end{equation}

If Condition~\ref{condProperEigen} holds, then the system of operators~\eqref{eqSystemB}
is linearly dependent (see~\cite[Sec.~3.1]{GurRJMP03}). Let
\begin{equation}\label{eqSystemB'}
\{\hat{\mathcal B}_{j'\sigma'\mu'}(D_y)\}
\end{equation}
be a maximal linearly independent subsystem of system~\eqref{eqSystemB}. In this case,
any operator $\hat{\mathcal B}_{j\sigma\mu}(D_y)$ which does not enter
system~\eqref{eqSystemB'} can be represented as follows:
\begin{equation}\label{eqBviaB'}
\hat{\mathcal
B}_{j\sigma\mu}(D_y)=\sum\limits_{j',\sigma',\mu'}\beta_{j\sigma\mu}^{j'\sigma'\mu'}\hat{\mathcal
B}_{j'\sigma'\mu'}(D_y),
\end{equation}
where $\beta_{j\sigma\mu}^{j'\sigma'\mu'}$ are some constants.

Introduce the notion of  consistency condition. Let $\{Z_{j\sigma\mu}\}\in\mathcal
W^{\bm}(\gamma^\varepsilon)$ be a vector of functions, each of which is defined on its
own interval $\gamma_{j\sigma}^\varepsilon$. Consider the functions
$$
Z^0_{j\sigma\mu}(r)=Z_{j\sigma\mu}(y)|_{y=(r\cos\omega_j,\, r(-1)^\sigma\sin\omega_j)}.
$$
Each of the functions $Z^0_{j\sigma\mu}$ belongs to $W^{\mjsigma}(0,\varepsilon)$.

\begin{definition}
Let $\beta_{j\sigma\mu}^{j'\sigma'\mu'}$ be the constants occurring in~\eqref{eqBviaB'}.
If the relations
\begin{equation}\label{eqConsistencyZ}
\int_{0}^\varepsilon
r^{-1}\Bigg|\frac{d^{2m-m_{j\sigma\mu}-1}}{dr^{2m-m_{j\sigma\mu}-1}}Z^0_{j\sigma\mu}-\sum\limits_{j',\sigma',\mu'}
\beta_{j\sigma\mu}^{j'\sigma'\mu'}\frac{d^{2m-m_{j'\sigma'\mu'}-1}}{dr^{2m-m_{j'\sigma'\mu'}-1}}Z^0_{j'\sigma'\mu'}\Bigg|^2dr<\infty
\end{equation}
hold for all indices $j,\sigma,\mu$ corresponding to the operators of
system~\eqref{eqSystemB} which do not enter system~\eqref{eqSystemB'}, then we say that
the {\em functions $Z_{j\sigma\mu}$ satisfy the consistency
condition~\eqref{eqConsistencyZ}}.
\end{definition}

\begin{remark}\label{remSufficCons}
The relation $\{Z_{j\sigma\mu}\}\in\mathcal H_0^{\bm}(\gamma^\varepsilon)$ is sufficient
(but not necessary) for the functions $Z_{j\sigma\mu}$ to satisfy the consistency
condition~\eqref{eqConsistencyZ}. This follows from Lemma~4.18 in~\cite{KondrTMMO67}.
\end{remark}

Now we will show that the following condition is necessary and sufficient for a given
generalized solution $u$ to belong to $W^{2m}(G)$.

\begin{condition}\label{condConsistencyPsi-BC}
Let $u$ be a generalized solution of problem~\eqref{eqPinG}, \eqref{eqBinG},
$\Psi_{j\sigma\mu}$ the right-hand sides in nonlocal conditions~\eqref{eqBinK}, and $W$
the polynomial vector appearing in Lemma~$\ref{lU=C+}$. Then the functions
$\Psi_{j\sigma\mu}-{\mathbf B}_{j\sigma\mu}(y,D_y)W$ satisfy the consistency
condition~\eqref{eqConsistencyZ}.
\end{condition}

\begin{theorem}\label{thuinW^2ProperEigen}
Let Conditions~$\ref{condProperEigen}$
and~$\ref{condNoEigenProperLambda}$--$\ref{condNoEigenMonom}$ hold, and let $u$ be a
generalized solution of problem~\eqref{eqPinG}, \eqref{eqBinG} with right-hand side
$\{f_0,f_{i\mu}\}\in L_2(G)\times \mathcal W^{\bm}(\partial G)$. Then $u\in W^{2m}(G)$ if
and only if Condition~$\ref{condConsistencyPsi-BC}$ holds.
\end{theorem}
\begin{proof} 1. {\em Necessity.} Let $u\in W^{2m}(G)$. Let the function $U=(U_1,\dots,U_N)$
correspond  to the set (orbit) $\mathcal K$. Clearly, $U\in\mathcal
W^{2m}(K^\varepsilon)$. It follows from Lemma~\ref{lU=C+} that $U=W+U'$, where
$U'\in\mathcal H_\delta^{2m}(K^\varepsilon)$, $0<\delta<1$. Since we additionally have
$U'=U-W\in\mathcal W^{2m}(K^\varepsilon)$, it follows from Sobolev's embedding theorem
that $D^\alpha U'(0)=0$, $|\alpha|\le 2m-2$. These relations and
Lemma~\ref{lAppL3.1GurRJMP03} imply that the functions $\Psi_{j\sigma\mu}-{\mathbf
B}_{j\sigma\mu}W={\mathbf B}_{j\sigma\mu}(y,D_y)U'$ satisfy the consistency
condition~\eqref{eqConsistencyZ}.

2. {\em Sufficiency.} Suppose that Condition~$\ref{condConsistencyPsi-BC}$ holds. It
follows from~\eqref{eqlU=C+00} and from Lemma~\ref{lAppL3.3GurRJMP03} that there exists a
function $ V\in\mathcal H_{\delta}^{2m}(K)\cap\mathcal W^{2m}(K) $ ($\delta$ is the same
as in Lemma~\ref{lU=C+}) such that
\begin{equation}\label{eqUinW^2_10'}
\begin{aligned}
\{\bP_j(y,D_y)(U_j'-V_j)\}&\in\cH_0^0(K^\varepsilon),\\
\{\bB_{j\sigma\mu}(y,D_y)(U' -V )\}&\in\cH_0^\bm(\gamma^\varepsilon).
\end{aligned}
\end{equation}

Due to~\eqref{eqUinW^2_10'} and the fact that the strip $1-2m\le \Im\lambda\le
1-2m+\delta$ contains only the proper eigenvalue $i(1-2m)$ of $\tilde{\mathcal
L}(\lambda)$, we can use Lemma~\ref{lAppL3.4GurRJMP03} to obtain that all the derivatives
of order $2m$ of the function $U'-V$ belong to $\mathcal W^0(K^\varepsilon)$. It follows
from this fact and from   the relations
$$
U'-V\in\mathcal H_{\delta}^{2m}(K^\varepsilon)\subset\mathcal
H_{0}^{2m-1}(K^\varepsilon)\subset\mathcal W^{2m-1}(K^\varepsilon)
$$
that $U'-V\in\mathcal W^{2m}(K^\varepsilon)$. Combining this relation with
Lemma~\ref{lU=C+}, we complete the proof of the sufficiency part.
\end{proof}

Note that Theorem~\ref{thuinW^2ProperEigen} enables us to conclude whether or not a
\textit{given} solution $u$ is smooth near the set $\mathcal K$, provided that we know
the asymptotics for $u$ of the kind~\eqref{eqlU=C+0} near the set $\mathcal K$ (i.e., if
we know the polynomial vector  $W$). Theorem~\ref{thuinW^2ProperEigen} shows what affects
the smoothness of solutions in principle. Below, this will enable us to obtain a
constructive condition which is necessary and sufficient for {\em any} generalized
solution to belong to $W^{2m}(G)$.

\subsection{Problem with Nonhomogeneous Nonlocal Conditions}\label{subsectNonHomogProb}

First of all, we show that the right-hand sides $f_{i\mu}$ in nonlocal
conditions~\eqref{eqBinG} must satisfy a certain consistency condition  in order that
generalized solutions be smooth.

Denote by $\mathcal S^{\bm}(\partial G)$ the set of functions $\{f_{i\mu}\}\in\mathcal
W^{\bm}(\partial G)$ such that the functions $F_{j\sigma\mu}$ (see~\eqref{eqytoy'})
satisfy the consistency condition~\eqref{eqConsistencyZ}. It follows from~Lemma~3.2
in~\cite{GurRJMP03} that the set $\mathcal S^{\bm}(\partial G)$ is not closed in the
space $\mathcal W^{\bm}(\partial G)$.

\begin{theorem}\label{thUNonSmFNonConsist}
Let  Conditions~$\ref{condProperEigen}$
and~$\ref{condNoEigenProperLambda}$--$\ref{condNoEigenMonom}$ hold. Then there exist a
function $\{f_0,f_{i\mu}\}\in L_2(G)\times \mathcal W^{\bm}(\partial G)$,
$\{f_{i\mu}\}\notin\mathcal S^{\bm}(\partial G)$, and a function $u\in W^{2m-1}(G)$ such
that $u$ is a generalized solution of problem~\eqref{eqPinG}, \eqref{eqBinG} with the
right-hand side $\{f_0,f_{i\mu}\}$ and $u\notin W^{2m}(G)$.
\end{theorem}

To prove Theorem~\ref{thUNonSmFNonConsist}, we preliminarily establish an auxiliary
result. Set
\begin{equation}\label{eqEpsilon'}
\varepsilon'=d_\chi\min(\varepsilon,\varkappa_2),
\end{equation}
where $d_\chi$ is defined in~\eqref{eqd1d2}.

\begin{lemma}\label{lUnonSmoothZnonConsist}
Let Condition~$\ref{condProperEigen}$ hold and a function $\{Z_{j\sigma\mu}\}\in\mathcal
W^{\bm}(\gamma^\varepsilon)$ be such that $\supp \{Z_{j\sigma\mu}\}\subset\mathcal
O_{\varepsilon/2}(0)$, $\dfrac{\partial^\beta}{\partial\tau_{j\sigma}^\beta}
Z_{j\sigma\mu}(0)=0$, $\beta\le 2m-m_{j\sigma\mu}-2$, and the functions $Z_{j\sigma\mu}$
do not satisfy the consistency condition~\eqref{eqConsistencyZ}. Then there exists a
function $U\in\mathcal H_\delta^{2m}(K)\subset\mathcal W^{2m-1}(K)$, $\delta>0$ is
arbitrary, such that $\supp U\subset\mathcal O_{\varepsilon'}(0)$, $U\notin\mathcal
W^{2m}(K^\varepsilon)$, and $U$ satisfies the relations
\begin{equation}\label{eqUnonSmoothZnonConsist}
\{{\mathbf P}_{j}(y,D_y)U_j\}\in\mathcal W^0(K^\varepsilon),\quad \{{\mathbf
B}_{j\sigma\mu}(y,D_y)U-Z_{j\sigma\mu}\}\in\mathcal H_0^{\bm}( \gamma^\varepsilon).
\end{equation}
\end{lemma}
\begin{proof} By Lemma~\ref{lDense}, there exists a sequence of functions
$\{Z_{j\sigma\mu}^n\}\in\mathcal W^{\bm}(\gamma)$, $n=1,2,\dots$, such that $\supp
Z_{j\sigma\mu}^n\subset\mathcal O_{\varepsilon}(0)$, $Z_{j\sigma\mu}^n$ vanish near the
origin (hence, they  satisfy the consistency condition~\eqref{eqConsistencyZ}), and
$\{Z_{j\sigma\mu}^n\}\to \{Z_{j\sigma\mu}\}$ in $W^{\bm}(\gamma)$. Taking into account
Lemma~\ref{lAppL2.1GurRJMP03}, we also see that $\{Z_{j\sigma\mu}^n\}\to
\{Z_{j\sigma\mu}\}$ in $H_\delta^{\bm}(\gamma)$, $\delta>0$ is arbitrary. Lemma~3.5
in~\cite{GurRJMP03}  ensures the existence of a sequence $V^n=(V_1^n,\dots,V_N^n)$
satisfying the following conditions: $V^n\in\mathcal W^{2m}(K^d)\cap \mathcal
H_\delta^{2m}(K^d)$ for any $d>0$,
\begin{equation}\label{eqPBVdeltainK}
{\mathcal P}_{j}(D_y)V_j^n=0 \quad (y\in K_j),\qquad {\mathcal B}_{j\sigma\mu}(D_y)V^n
=Z_{j\sigma\mu}^n(y) \quad (y\in\gamma_{j\sigma}),
\end{equation}
and the sequence $V^n$ converges to a function $V\in\mathcal H_\delta^{2m}(K^d)$ in
$\mathcal H_\delta^{2m}(K^d)$ for any $d>0$. Passing to the limit
in~\eqref{eqPBVdeltainK} (in the spaces $\cH_\delta^0(K^d)$ and $\cH_\delta^{\bm}(K^d)$,
respectively), we obtain
\begin{equation}\label{eqPBVinK}
{\mathcal P}_{j}(D_y)V_j=0 \quad (y\in K_j),\qquad {\mathcal B}_{j\sigma\mu}(D_y)V
=Z_{j\sigma\mu}(y) \quad (y\in\gamma_{j\sigma}).
\end{equation}

Consider a cut-off function $\xi\in C_0^\infty(\cO_{\varepsilon'}(0))$ equal to one near
the origin. Set $U=\xi V$. Clearly, $\supp U\subset\mathcal O_{\varepsilon'}(0)$ and
$$
U\in \mathcal H_\delta^{2m}(K)\subset\mathcal W^{2m-1}(K).
$$

2. We claim that $U$ is the desired function. Indeed, using Leibniz' formula,
relations~\eqref{eqPBVinK} and Lemma~\ref{lAppL3.3'Kondr}, we
infer~\eqref{eqUnonSmoothZnonConsist}.

It remains to prove that $U\notin \mathcal W^{2m}(K^\varepsilon)$. Assume the contrary.
Let $U\in \mathcal W^{2m}(K^\varepsilon)$. In this case, it follows from Sobolev's
embedding theorem and from the belonging $U\in\mathcal H_\delta^{2m}(K^\varepsilon)$
($\delta>0$ is arbitrary) that $D^\alpha U(0)=0$, $|\alpha|\le 2m-2$. Combining this fact
with Lemma~\ref{lAppL3.1GurRJMP03} implies that the functions ${\mathbf
B}_{j\sigma\mu}(y,D_y)U$ satisfy the consistency condition~\eqref{eqConsistencyZ}.
However, the functions ${\mathbf B}_{j\sigma\mu}(y,D_y)U-Z_{j\sigma\mu}$ do not satisfy
the consistency condition~\eqref{eqConsistencyZ} in that case. This
contradicts~\eqref{eqUnonSmoothZnonConsist} (see Remark~\ref{remSufficCons}).
\end{proof}

{\bf Proof of Theorem~$\ref{thUNonSmFNonConsist}$.} 1. We will
construct a generalized solution $u\notin W^{2m}(G)$ supported
near the set $\mathcal K$ so that $\mathbf B^2_{i\mu}u=0$ due
to~\eqref{eqSeparK23'}.

It was shown in the course of the proof of Lemma~3.2 in~\cite{GurRJMP03} that there
exists a function $\{Z_{j\sigma\mu}\}\in\mathcal W^{\bm}(\gamma)$ such that $\supp
Z_{j\sigma\mu}\subset\mathcal O_{\varepsilon/2}(0)$,
$\dfrac{\partial^\beta}{\partial\tau_{j\sigma}^\beta} Z_{j\sigma\mu}(0)=0$, $\beta\le
2m-m_{j\sigma\mu}-2$, and the functions $Z_{j\sigma\mu}$ do not satisfy the consistency
condition~\eqref{eqConsistencyZ}. By Lemma~\ref{lUnonSmoothZnonConsist}, there exists a
function $U\in \mathcal H_\delta^{2m}(K)\subset\mathcal W^{2m}(K)$ such that $\supp
U\subset\mathcal O_{\varepsilon'}(0)$, $U\notin\mathcal W^{2m}(K)$, and $U$ satisfies
relations~\eqref{eqUnonSmoothZnonConsist}. Therefore, $\{{\mathbf P}_{j}(y,D_y)U_j\}\in
\mathcal W^0(K^\varepsilon)$, $\{{\mathbf B}_{j\sigma\mu}(y,D_y)U\}\in\mathcal
W^{\bm}(\gamma^\varepsilon)$, and the functions ${\mathbf B}_{j\sigma\mu}(y,D_y)U$ do not
satisfy the consistency condition~\eqref{eqConsistencyZ}.

2. Introduce a function $u(y)$ such that $u(y)=U_j(y'(y))$ for $y\in\mathcal
O_{\varepsilon'}(g_j)$ and $u(y)=0$ for $y\notin\mathcal O_{\varepsilon'}(\mathcal K)$,
where $y'\mapsto y(g_j)$ is the change of variables inverse to the change of variables
$y\mapsto y'(g_j)$ from Sec.~\ref{subsectStatement}. Since $\supp u\subset \mathcal
O_{\varkappa_1}(\mathcal K)$, it follows that $\mathbf B_{i\mu}^2u=0$. Therefore, $u(y)$
is the desired generalized solution of problem~\eqref{eqPinG}, \eqref{eqBinG}.
\endproof

\bigskip

Theorem~\ref{thUNonSmFNonConsist} shows that  if one wants that {\em any} generalized
solution of problem~\eqref{eqPinG}, \eqref{eqBinG} be smooth, then one must take
right-hand sides $\{f_0,f_{i\mu}\}$ from the space $L_2(G)\times \mathcal
S^{\bm}(\partial G)$.

\medskip

Let $v$ be an arbitrary function from the space $W^{2m}(G\setminus\overline{\mathcal
O_{\varkappa_1}(\mathcal K)})$. Consider the change of variables $y\mapsto y'(g_j)$ from
Sec.~\ref{subsectStatement}  and introduce the functions
\begin{equation}\label{eqBjsigmav}
B^v_{j\sigma\mu}(y')=(\mathbf B_{i\mu}^2v)(y(y')),\quad y'\in\gamma_{j\sigma}^\varepsilon
\end{equation}
(cf.~\eqref{eqytoy'}). We prove that the following condition is necessary and sufficient
for any generalized solution to be smooth.
\begin{condition}\label{condB2vB1CConsistency}
\begin{enumerate}
\item For any $v\in
W^{2m}(G\setminus\overline{\mathcal O_{\varkappa_1}(\mathcal K)})$, the functions
$B^v_{j\sigma\mu}$ satisfy the consistency condition~\eqref{eqConsistencyZ}.
\item For any polynomial vector $W$ of degree $2m-2$
the functions $\mathbf B_{j\sigma\mu}(y,D_y)W$ satisfy the consistency
condition~\eqref{eqConsistencyZ}.
\end{enumerate}
\end{condition}

Note that the validity of Condition~\ref{condB2vB1CConsistency}, unlike
Condition~\ref{condConsistencyPsi-BC}, does not depend on a generalized solution. It
depends only on the operators $\mathbf B_{i\mu}^1$ and $\mathbf B_{i\mu}^2$ and on the
geometry of the domain $G$ near the set (orbit) $\mathcal K$. This is quite natural
because we study the smoothness of {\em all} generalized solutions in this section (while
in Sec.~\ref{subsectuFixed}, we have investigated the smoothness of a fixed solution).

\begin{theorem}\label{thSmoothfne0}
Let Conditions~$\ref{condProperEigen}$
and~$\ref{condNoEigenProperLambda}$--$\ref{condNoEigenMonom}$ hold.
\begin{enumerate}
\item
If Condition~$\ref{condB2vB1CConsistency}$ holds and $u$ is a generalized solution of
problem~\eqref{eqPinG}, \eqref{eqBinG} with right-hand side $\{f_0,f_{i\mu}\}\in
L_2(G)\times \mathcal S^{\bm}(\partial G)$, then $u\in W^{2m}(G)$.
\item
If Condition~$\ref{condB2vB1CConsistency}$ fails, then there exists a right-hand side
$\{f_0,f_{i\mu}\}\in L_2(G)\times \mathcal S^{\bm}(\partial G)$ and a generalized
solution $u$  of problem~\eqref{eqPinG}, \eqref{eqBinG}  such that $u\notin W^{2m}(G)$.
\end{enumerate}
\end{theorem}
\begin{proof} 1. {\em Sufficiency.} Let Condition~\ref{condB2vB1CConsistency} hold, and let $u$ be
an arbitrary generalized solution of problem~\eqref{eqPinG}, \eqref{eqBinG} with
right-hand side $\{f_0,f_{i\mu}\}\in L_2(G)\times \mathcal S^{\bm}(\partial G)$.
By~\eqref{eqSmoothOutsideK}, we have $u\in W^{2m}(G\setminus\overline{\mathcal
O_{\varkappa_1}(\mathcal K)})$. Therefore, by Condition~\ref{condB2vB1CConsistency}, the
functions $B^u_{j\sigma\mu}$ satisfy the consistency condition~\eqref{eqConsistencyZ}.
Let $W$ be a polynomial vector of degree $2m-2$ defined by Lemma~\ref{lU=C+}. Using
Condition~\ref{condB2vB1CConsistency} again, we see that the functions $\mathbf
B_{j\sigma\mu}(y,D_y)W$ satisfy the consistency condition~\eqref{eqConsistencyZ}. Since
$\{f_{i\mu}\}\in\mathcal S^{\bm}(\partial G)$, it follows that the functions
$F_{j\sigma\mu}$ satisfy the consistency condition~\eqref{eqConsistencyZ}. Therefore, the
functions $\Psi_{j\sigma\mu}=F_{j\sigma\mu}-B^u_{j\sigma\mu}$ and $\mathbf
B_{j\sigma\mu}(y,D_y)W$ satisfy Condition~\ref{condConsistencyPsi-BC}. Applying
Theorem~\ref{thuinW^2ProperEigen}, we obtain $u\in W^{2m}(G)$.

2. {\em Necessity.} Let Condition~\ref{condB2vB1CConsistency} fail. In this case, there
exist a function $v\in W^{2m}(G\setminus\overline{\mathcal O_{\varkappa_1}(\mathcal K)})$
and a polynomial vector $W=(W_1,\dots,W_N)$ of degree $2m-2$ such that the functions
$B^v_{j\sigma\mu}+\mathbf B_{j\sigma\mu}W$ do not satisfy the consistency
condition~\eqref{eqConsistencyZ} (one can assume that either $v=0,\,W\ne0$ or
$v\ne0,\,W=0$). Extend the function $v$ to the domain $G$ in such a way that $v(y)=0$ for
$y\in\mathcal O_{\varkappa_1/2}(\mathcal K)$ and $v\in W^{2m}(G)$.

By Lemma~4.11 in~\cite{KondrTMMO67}, there exist polynomials $F_{j\sigma\mu}'(r)$ of
degree $2m-m_{j\sigma\mu}-2$   such that
$$
\{B_{j\sigma\mu}^v+\bB_{j\sigma\mu}(y,D_y)W-F_{j\sigma\mu}'\}\in
\cH_\delta^{\bm}(\gamma^\varepsilon)\cap\cW^{\bm}(\gamma^\varepsilon),
$$
where $\delta>0$ is arbitrary. Hence,
$$
\dfrac{\partial^\beta}{\partial\tau_{j\sigma}^\beta}(B_{j\sigma\mu}^v+\bB_{j\sigma\mu}(y,D_y)W-F_{j\sigma\mu}')(0)=0,\qquad
\beta\le 2m-m_{j\sigma\mu}-2.
$$

 Since
$\dfrac{d^{2m-m_{j\sigma\mu}-1}}{dr^{2m-m_{j\sigma\mu}-1}}F_{j\sigma\mu}'(r)\equiv0$, it
follows that the functions $F_{j\sigma\mu}'$ satisfy the consistency
condition~\eqref{eqConsistencyZ}. Therefore, the functions
$B_{j\sigma\mu}^v+\bB_{j\sigma\mu}(y,D_y)W-F_{j\sigma\mu}'$ do not satisfy the
consistency condition~\eqref{eqConsistencyZ}.

By Lemma~\ref{lUnonSmoothZnonConsist}, there exists a function $U'\in\mathcal
H_\delta^{2m}(K)\subset\mathcal W^{2m-1}(K)$ such that $\supp U'\subset\mathcal
O_{\varepsilon'}(0)$, $U'\notin\mathcal W^{2m}(K^\varepsilon)$, and
\begin{equation}\label{eqPU'-PsiinH0}
\{\mathbf P_j(y,D_y)U'_j\}\in\mathcal W^0(K^\varepsilon),
\end{equation}
$$
\{\mathbf B_{j\sigma\mu}(y,D_y)U'-(F_{j\sigma\mu}'-B^v_{j\sigma\mu}-\mathbf
B_{j\sigma\mu}(y,D_y)W))\}\in\mathcal H_0^{\bm}(\gamma^\varepsilon).
$$
One can also write the latter relation as follows:
\begin{equation}\label{eqBU'-PsiinH0}
\{\mathbf B_{j\sigma\mu}(y,D_y)(U'+W)+B^v_{j\sigma\mu}-F_{j\sigma\mu}'\}\in\mathcal
H_0^{\bm}(\gamma^\varepsilon).
\end{equation}
Introduce a function $u'(y)$ such that $u'(y)=U'_j(y'(y))+\xi_j(y)W_j$ for $y\in\mathcal
O_{\varepsilon'}(g_j)$ and $u'(y)=0$ for $y\notin\mathcal O_{\varepsilon'}(\mathcal K)$,
where $y'\mapsto y(g_j)$ is the change of variables inverse to the change of variables
$y\mapsto y'(g_j)$ from Sec.~\ref{subsectStatement}, while $\xi_j\in
C_0^\infty(O_{\varepsilon'}(g_j))$, $\xi_j(y)=1$ for $y\in\mathcal
O_{\varepsilon'/2}(g_j)$, and $\varepsilon'$ is given by~\eqref{eqEpsilon'}. Let us prove
that the function $u=u'+v$ is the desired one. Clearly, $u\in W^{2m-1}(G)$, $u\notin
W^{2m}(G)$, and $u$ satisfies relations~\eqref{eqSmoothOutsideK}. It follows from the
belonging $v\in W^{2m}(G)$ and from relations~\eqref{eqPU'-PsiinH0} that
$$
\mathbf P(y,D_y)u\in L_2(G).
$$
Consider the functions $f_{i\mu}=\bB_{i\mu}^0u+\mathbf B_{i\mu}^1 u+\mathbf B_{i\mu}^2
u$. It follows from the belonging $v\in W^{2m}(G)$, from
relations~\eqref{eqSmoothOutsideK}, and from inequality~\eqref{eqSeparK23'} that
$f_{i\mu}\in W^{\mimu}\bigl(\Gamma_i\setminus\overline{\mathcal O_\delta(\mathcal
K)}\bigr)$ for any $\delta>0$. Consider the behavior of $f_{i\mu}$ near the set $\mathcal
K$. Note that $\mathbf B_{i\mu}^2 u'=0$ by~\eqref{eqSeparK23'}. Furthermore,
$\bB_{i\mu}^0 v+\mathbf B_{i\mu}^1 v=0$ for $y\in\mathcal O_{\varkappa_1/D_\chi
}(\mathcal K)$. Therefore,
\begin{equation}\label{eqfi}
 f_{i\mu}=\bB_{i\mu}^0 u'+\mathbf B_{i\mu}^1 u'+\mathbf B_{i\mu}^2 v\quad (y\in\mathcal
O_{\varkappa_1/D_\chi }(\mathcal K)).
\end{equation}
Introduce the functions $F_{j\sigma\mu}(y')=f_{i\mu}(y(y'))$, where $y\mapsto y'(g_j)$ is
the change of variables from Sec.~\ref{subsectStatement}. It follows from~\eqref{eqfi}
and from~\eqref{eqBU'-PsiinH0} that $\{F_{j\sigma\mu}-F_{j\sigma\mu}'\}\in\mathcal
H_0^{\bm}(\gamma^\varepsilon)$. Therefore, $\{F_{j\sigma\mu}\}\in\mathcal
W^{\bm}(\gamma^\varepsilon)$ and the functions $F_{j\sigma\mu}$, together with
$F_{j\sigma\mu}'$, satisfy the consistency condition~\eqref{eqConsistencyZ}. Hence
$\{f_{i\mu}\}\in\mathcal S^{\bm}(\partial G)$, which completes the proof.
\end{proof}

\subsection{Problem with Regular Nonlocal Conditions}

\begin{definition}\label{defAdmit}
We say that a function $v\in W^{2m}(G\setminus\overline{\mathcal O_{\varkappa_1}(\mathcal
K)})$ is {\em admissible} if there exists a polynomial vector $W=(W_1,\dots,W_N)$ of
degree $2m-2$ such that
\begin{equation}\label{eqvadmissible}
\dfrac{\partial^\beta}{\partial\tau_{j\sigma}^\beta}(B_{j\sigma\mu}^v+\mathbf
B_{j\sigma\mu}(y,D_y)W)(0)=0,\quad \beta\le 2m-m_{j\sigma\mu}-2.
\end{equation}
Any polynomial vector $W$ of degree $2m-2$  satisfying relations~\eqref{eqvadmissible} is
said to be an {\em admissible polynomial vector corresponding to the function~$v$.}
\end{definition}

Let $\tau_{gi}$ be the unit vector parallel to $\Gamma_i$ near the point
$g\in\overline{\Gamma_i}\cap\cK$.

\begin{definition}\label{defRerularRHS}
\begin{enumerate}
\item
The right-hand sides $f_{i\mu}$ in nonlocal conditions~\eqref{eqBinG} are said to be
\textit{regular} if $\{f_{i\mu}\}\in\mathcal W^{\bm}(\partial G)$ and
$$
\dfrac{\partial^\beta}{\partial\tau_{gi}^\beta} f_{i\mu}(g)=0,\qquad \beta\le
2m-m_{i\mu}-2,\ g\in\overline{\Gamma_i}\cap K.
$$
\item
The right-hand sides $\Psi_{j\sigma\mu}$ in nonlocal conditions~\eqref{eqBinK} are said
to be \textit{regular} if $\{\Psi_{j\sigma\mu}\}\in\mathcal W^{\bm}(\gamma^\varepsilon)$
and
$$
\dfrac{\partial^\beta}{\partial\tau_{j\sigma}^\beta} \Psi_{j\sigma\mu}(0)=0,\qquad
\beta\le 2m-m_{j\sigma\mu}-2.
$$
\end{enumerate}
If $m_{i\mu}=2m-1$ or $m_{j\sigma\mu}=2m-1$, then the corresponding relations are absent.
\end{definition}

In particular, the right-hand sides $\{f_{i\mu}\}\in\mathcal H_0^{\bm}(\partial G)$ and
$\{\Psi_{j\sigma\mu}\}\in\mathcal H_0^{\bm}(\gamma^\varepsilon)$ are regular due to
Sobolev's embedding theorem. In this subsection, we prove that the following condition
(which is weaker than Condition~\ref{condB2vB1CConsistency}) is necessary and sufficient
for any generalized solution of problem~\eqref{eqPinG}, \eqref{eqBinG} with regular
right-hand sides $\{f_{i\mu}\}\in\mathcal S^{\bm}(\partial G)$ to be smooth.

\begin{condition}\label{condBv+BCConsist}
For each admissible function $v$ and   each admissible polynomial vector $W$ {\rm (}of
degree $2m-2${\rm )} corresponding to $v$, the functions $B_{j\sigma\mu}^v+\mathbf
B_{j\sigma\mu}(y,D_y)W$ satisfy the consistency condition~\eqref{eqConsistencyZ}.
\end{condition}

\begin{theorem}\label{thSmoothf0}
Let Conditions~$\ref{condProperEigen}$
and~$\ref{condNoEigenProperLambda}$--$\ref{condNoEigenMonom}$ hold.
\begin{enumerate}
\item
If Condition~$\ref{condBv+BCConsist}$ holds and $u$  is a generalized solution of
problem~\eqref{eqPinG}, \eqref{eqBinG} with right-hand side $\{f_0,f_{i\mu}\}\in
L_2(G)\times\mathcal S^{\bm}(\partial G)$, where $f_{i\mu}$ are regular, then $u\in
W^{2m}(G)$.
\item
If Condition~$\ref{condBv+BCConsist}$ fails, then there exists a right-hand side
$\{f_0,f_{i\mu}\}\in L_2(G)\times\mathcal H_0^{\bm}(\partial G)$ and a generalized
solution $u$ of problem~\eqref{eqPinG}, \eqref{eqBinG} such that $u\notin W^{2m}(G)$.
\end{enumerate}
\end{theorem}
\begin{proof} 1. {\em Sufficiency.} Let Condition~\ref{condBv+BCConsist} hold, and let $u$ be an
arbitrary generalized solution of problem~\eqref{eqPinG}, \eqref{eqBinG} with right-hand
side $\{f_0,f_{i\mu}\}\in L_2(G)\times\mathcal S^{\bm}(\partial G)$, where $f_{i\mu}$ are
regular. By~\eqref{eqSmoothOutsideK}, we have $u\in W^{2m}(G\setminus\overline{\mathcal
O_{\varkappa_1}(\mathcal K)})$.

It follows from the properties of $f_{i\mu}$ that the right-hand sides in nonlocal
conditions~\eqref{eqBinK} have the form
\begin{equation}\label{eqPsi=B^u}
\Psi_{j\sigma\mu}=F_{j\sigma\mu}-B_{j\sigma\mu}^u,
\end{equation}
where $\{F_{j\sigma\mu}\}\in\mathcal W^{\bm}(\gamma^\varepsilon)$,
\begin{equation}\label{eqSmoothf0*}
\dfrac{\partial^\beta}{\partial\tau_{j\sigma}^\beta}F_{j\sigma\mu}(0)=0,\quad \beta\le
2m-m_{j\sigma\mu}-2,
\end{equation}
and $F_{j\sigma\mu}$ satisfy the consistency condition~\eqref{eqConsistencyZ}.

 Further, let $U=W+U'$, where
$U'\in\mathcal H_{\delta}^{2m}(K^\varepsilon)$ and $W$ are the function and the
polynomial vector (of degree $2m-2$) defined in Lemma~\ref{lU=C+}. It follows
from~\eqref{eqBinK} and~\eqref{eqPsi=B^u} that
$$
{\mathbf B}_{j\sigma\mu}(y,D_y)U'= F_{j\sigma\mu}-(B_{j\sigma\mu}^u+{\mathbf
B}_{j\sigma\mu}(y,D_y)W).
$$
Since $\{B_{j\sigma\mu}^u+{\mathbf B}_{j\sigma\mu}(y,D_y)W-F_{j\sigma\mu}\}\in\mathcal
W^{\bm}(\gamma^\varepsilon)$ and $U'\in \mathcal H_{\delta}^{2m}(K^\varepsilon)$, it
follows that
$$
\begin{aligned}
&\{B_{j\sigma\mu}^u+ {\mathbf B}_{j\sigma\mu}(y,D_y)W -F_{j\sigma\mu}\}\\
&\qquad=\{-{\mathbf B}_{j\sigma\mu}(y,D_y)U'\}\in\mathcal
H_\delta^{\bm}(\gamma^\varepsilon)\cap \mathcal W^{\bm}(\gamma^\varepsilon).
\end{aligned}
$$
It follows from this relation and from~\eqref{eqSmoothf0*} that
$$
\dfrac{\partial^\beta}{\partial\tau_{j\sigma}^\beta}(B_{j\sigma\mu}^u+ {\mathbf
B}_{j\sigma\mu}(y,D_y)W)(0)=0,\quad \beta\le 2m-m_{j\sigma\mu}-2,
$$
i.e.,  $u$ is an admissible function and $W$ is an admissible polynomial vector
corresponding to $u$. Hence, by virtue of~\eqref{eqPsi=B^u} and by
Condition~\ref{condBv+BCConsist}, Condition~\ref{condConsistencyPsi-BC} holds. Combining
this fact with Theorem~\ref{thuinW^2ProperEigen} implies $u\in W^{2m}(G)$.

2. {\em Necessity.} Let Condition~\ref{condBv+BCConsist} fail. In this case, there exists
a function $v\in W^{2m}(G\setminus\overline{\mathcal O_{\varkappa_1}(\mathcal K)})$ and a
polynomial vector $W=(W_1,\dots,W_N)$ of degree $2m-2$ such that
$$
\dfrac{\partial^\beta}{\partial\tau_{j\sigma}^\beta}(B_{j\sigma\mu}^u+ {\mathbf
B}_{j\sigma\mu}(y,D_y)W)(0)=0,\quad \beta\le 2m-m_{j\sigma\mu}-2,
$$
 and the
functions $B^v_{j\sigma\mu}+\mathbf B_{j\sigma\mu}(y,D_y)W$ do not satisfy the
consistency condition~\eqref{eqConsistencyZ}.

We must find a function  $u\in W^\ell(G)$ satisfying relations~\eqref{eqSmoothOutsideK}
such that $u\notin W^{2m}(G)$ and
$$
\mathbf P(y,D_y)u\in L_2(G),\qquad \{\bB_{i\mu}^0 u +\mathbf B_{i\mu}^1 u+\mathbf
B_{i\mu}^2 u\}\in \cH_0^{\bm}(\pG).
$$
To do this, one can repeat the proof of assertion~2 of Theorem~\ref{thSmoothfne0},
assuming that $v$ is the above function, $W$ is the above polynomial vector, and
$F_{j\sigma\mu}'(y)\equiv0$ (which is possible due to the relation
$B^v_{j\sigma\mu}+\mathbf
B_{j\sigma\mu}(y,D_y)W\in\cH_\delta^\bm(\gamma^\varepsilon)\cap\cW^\bm(\gamma^\varepsilon)$,
where $\delta>0$ is arbitrary).
\end{proof}

\section{Violation of Smoothness of Generalized Solutions}\label{sectImproperEigen}

\subsection{Violation of Conditions~\ref{condNoEigen1-2m} and~\ref{condProperEigen} or
Condition~\ref{condNoEigenProperLambda}}

The title of this subsectoin means that the following condition holds.

\begin{condition}\label{condImproperEigen}
The band $1-2m\le\Im\lambda<1-\ell$ contains an improper eigenvalue of the
operator~$\tilde{\mathcal L}(\lambda)$.
\end{condition}
We show that the smoothness of generalized solutions can be violated for any operators
$\mathbf B_{i\mu}^2$.

\begin{theorem}\label{thNonSmoothu}

Let Condition~$\ref{condImproperEigen}$ hold. Then there exists a right-hand side
$\{f_0,f_{i\mu}\}\in L_2(G)\times\cH_0^{\bm}(\pG)$ and a generalized solution $u$ of
problem~\eqref{eqPinG}, \eqref{eqBinG} such that $u\notin W^{2m}(G)$.
\end{theorem}
\begin{proof} 1. Let $\lambda=\lambda_0$ be an improper eigenvalue of the operator $\tilde{\mathcal
L}(\lambda)$, $1-2m\le\Im\lambda_0<1-\ell$. Consider the function
\begin{equation}\label{eq5.1'}
 V=r^{i\lambda_0}\sum\limits_{l=0}^{l_0}\frac{1}{l!}(i\ln
 r)^l\varphi^{(l_0-l)}(\omega)\in \cW^\ell(K^d)\quad \forall d>0,
\end{equation}
where $\varphi^{(0)}, \dots, \varphi^{(\varkappa-1)}$ are an eigenvector and associated
vectors (a Jordan chain of length $\varkappa\ge 1$) of the operator $\tilde{\mathcal
L}(\lambda)$ corresponding to the eigenvalue $\lambda_0$. The number $l_0$ ($0\le
l_0\le\varkappa-1$) occurring in the definition of $V$ is such that the function $V$ is
not a polynomial vector in $y_1, y_2$. Such a number $l_0$ does exist because $\lambda_0$
is not a proper eigenvalue (if $\Im\lambda$ is a noninteger  or $\Im\lambda$ is an
integer but $\Re\lambda\ne0$, then we can take $l_0=0$).

Since $V$ is not a polynomial vector, it follows from Lemma~\ref{lAppL4.20Kondr} that
\begin{equation}\label{eq5.5}
V\notin\cW^{2m}(K^d)\quad \forall d>0.
\end{equation}
It follows from Lemma~\ref{lAppL2.1GurPetr03} that
\begin{equation}\label{eq5.2}
 {\mathcal P}_j(D_y)V_j=0,\qquad {\mathcal B}_{j\sigma\mu}(D_y)V|_{\gamma_{j\sigma}}=0.
\end{equation}

Using~\eqref{eq5.2} and the Taylor expansion for the coefficients of $\mathbf P_j(y,D_y)$
and $\mathbf B_{j\sigma\mu}(y,D_y)$, we have
\begin{equation}\label{eq5.2'}
\{{\mathbf P}_j(y,D_y)V_j-P_j\}\in\mathcal W^0(K^\varepsilon),\quad \{{\mathbf
B}_{j\sigma\mu}(y,D_y)V-P_{j\sigma\mu}\}\in\mathcal H_0^{\bm}(\gamma^\varepsilon),
\end{equation}
where $P_j$ is a linear combination of terms of the  kind
$$
r^{i\lambda_0-2m+1}(i\ln r)^l\varphi(\omega),\dots, r^{i\lambda_0-2m+k_0}(i\ln
r)^l\varphi(\omega),
$$
 $P_{j\sigma\mu}$ is a linear combination of terms of the kind
$$
r^{i\lambda_0-m_{j\sigma\mu}+1}(i\ln r)^l,\dots, r^{i\lambda_0-m_{j\sigma\mu}+k_0}(i\ln
r)^l,
$$
$\varphi(\omega)$ are infinitely smooth vector-valued functions, and $k_0\in\bbN$ is such
that
\begin{equation}\label{eqk_0}
-\Im\lambda_0-2m+k_0\le -1,\quad -\Im\lambda_0-2m+k_0+1> -1.
\end{equation}
Clearly, one can set $P_j=0$ and $P_{j\sigma\mu}=0$ if inequalities~\eqref{eqk_0} are
true for $k_0=0$, i.e., if $1-2m\le\Im\lambda_0<2-2m$.

Using Lemma~\ref{lAppL4.3GurPetr03}, we can construct the function
\begin{equation}\label{eq5.2''}
V'=\sum\limits_{k=1}^{k_0}\sum\limits_{l=0}^{l'} r^{i\lambda_0+k}(i\ln
r)^{l_k}\varphi_{kl}(\omega)\in W^\ell(K^d)\quad \forall d>0
\end{equation}
such that
\begin{equation}\label{eq5.2'''}
\{{\mathbf P}_j(y,D_y)V_j'-P_j\}\in\mathcal W^0(K^\varepsilon),\qquad \{{\mathbf
B}_{j\sigma\mu}(y,D_y)V'-P_{j\sigma\mu}\}\in\mathcal H_0^{\bm}(\gamma^\varepsilon).
\end{equation}

Consider a cut-off function $\xi\in C_0^\infty(\cO_{\varepsilon'}(0))$ equal to one near
the origin, where $\varepsilon'$ is given by~\eqref{eqEpsilon'}. Set $U=\xi (V-V')$.
Clearly, $\supp U\subset\cO_{\varepsilon'}(0)$; hence,
\begin{equation}\label{eqNonSmoothu3'}
\supp{\mathbf B}_{j\sigma\mu}(y,D_y)U\subset\overline{\gamma_{j\sigma}}\cap\mathcal
O_{\varkappa_2}(0).
\end{equation}

 It follows from~\eqref{eq5.1'}, \eqref{eq5.2''},
and~\eqref{eq5.5} that
\begin{equation}\label{eq5.5*}
U\in \cW^\ell(K), \qquad U\notin\cW^{2m}(K^d)\quad\forall d>0.
\end{equation}
Moreover, by virtue of~\eqref{eq5.2'} and~\eqref{eq5.2'''}, we have
\begin{equation}\label{eq5.2''''}
\{{\mathbf P}_j(y,D_y)U_j\}\in\mathcal W^0(K^\varepsilon),\qquad \{{\mathbf
B}_{j\sigma\mu}(y,D_y)U\}\in\mathcal H_0^{\bm}(\gamma^\varepsilon).
\end{equation}

2. Consider the function $u(y)$ given by $u(y)=U_j(y'(y))$ for $y\in\mathcal
O_{\varepsilon'}(g_j)$ and $u(y)=0$ for $y\notin\mathcal O_{\varepsilon'}(\mathcal K)$,
where $y'\mapsto y(g_j)$ is the change of variables inverse to the change of variables
$y\mapsto y'(g_j)$ from Sec.~\ref{subsectStatement}. The function $u$ is the desired one.
Indeed, $u\notin W^{2m}(G)$ due to~\eqref{eq5.5*}. Furthermore, $\mathbf B_{i\mu}^2u=0$
due to inequality~\eqref{eqSeparK23'} because $\supp u\subset \mathcal
O_{\varkappa_1}(\mathcal K)$. It follows from the equality $\mathbf B_{i\mu}^2u=0$ and
from relations~\eqref{eq5.2''''} that the function $u$ satisfies the following relations:
\begin{equation}\label{eqNonSmoothu1}
\begin{gathered}
\mathbf P(y,D_y)u\in L_2(G),\qquad \mathbf B_{i\mu}^0 u+\mathbf B_{i\mu}^1
u+\mathbf B_{i\mu}^2 u\in H_0^{\mimu}(\Gamma_i),\\
\supp(\mathbf B_{i\mu}^0 u+\mathbf B_{i\mu}^1u+\mathbf
B_{i\mu}^2u)\subset\overline{\Gamma_i}\cap\mathcal O_{\varkappa_2}(\mathcal K).
\end{gathered}
\end{equation}
\end{proof}

\subsection{Violation of Condition~\ref{condEigenMonom}
or~\ref{condNoEigenMonom}}\label{subsecNonSmoothuNotPolyn}

If $\ell=2m-1$, then all the possibilities for the location of eigenvalues of
$\tilde\cL(\lambda)$ have been investigated. It remains to assume that $\ell\le 2m-2$ and
Condition~\ref{condEigenMonom} or~\ref{condNoEigenMonom} fails.

\begin{theorem}\label{thNonSmoothuNotPolyn}
Suppose that Condition~$\ref{condNoEigenProperLambda}$ holds while
Condition~$\ref{condEigenMonom}$ or~$\ref{condNoEigenMonom}$ fails. Then there is a
right-hand side $\{f_0,f_{i\mu}^1+f_{i\mu}^2\}\in L_2(G)\times\cW^{\bm}(\pG)$ and a
generalized solution $u$ of problem~\eqref{eqPinG}, \eqref{eqBinG} such that $u\notin
W^{2m}(G)$, where $f_{i\mu}^1$ is a polynomial of degree $ 2m-m_{i\mu}-2$ in a
neighborhood of the point $g\in\overline{\Gamma_i}\cap\cK$ and $\{f_{i\mu}^2\}\in
\cH_0^{\bm}(\pG)$.
\end{theorem}
\begin{proof} 1. Due to part 2 of Remark~\ref{remMonomialSolution}, there is a function $V$ given
by~\eqref{eqUinW^2_4''} such that
\begin{equation}\label{eqNonSmoothuNotPolyn1}
 V\in \cW^\ell(K^d), \quad V\notin \cW^{2m}(K^d)\quad \forall d>0,
\end{equation}
\begin{equation}\label{eqNonSmoothuNotPolyn2}
 {\mathcal P}_j(D_y)V_j=0,\qquad {\mathcal B}_{j\sigma\mu}(D_y)V|_{\gamma_{j\sigma}}=c_{j\sigma\mu}
 r^{s-m_{j\sigma\mu}}
\end{equation}
for some $s\in\{\ell,\dots,2m-2\}$ and some (nontrivial) numerical vector
$\{c_{j\sigma\mu}\}\in C_s$.

Using~\eqref{eqNonSmoothuNotPolyn2} and the   Taylor expansion for the coefficients of
$\mathbf P_j(y,D_y)$ and $\mathbf B_{j\sigma\mu}(y,D_y)$, we have
\begin{equation}\label{eqNonSmoothuNotPolyn3}
\begin{aligned}
\{{\mathbf P}_j(y,D_y)V_j-P_j\}&\in\mathcal W^0(K^\varepsilon),\\
 \{{\mathbf
B}_{j\sigma\mu}(y,D_y)V-c_{j\sigma\mu}
 r^{s-m_{j\sigma\mu}}-P_{j\sigma\mu}\}&\in\mathcal
H_0^{\bm}(\gamma^\varepsilon),
\end{aligned}
\end{equation}
where the functions $P_j$ and $P_{j\sigma\mu}$ are of the same form as in~\eqref{eq5.2'}.

As in the proof of Theorem~\ref{thNonSmoothu}, we can construct a function $V'$ of the
form~\eqref{eq5.2''} (with $i\lambda_0$ replaced by $s$) satisfying
relations~\eqref{eq5.2'''}.

Consider a cut-off function $\xi\in C_0^\infty(\cO_{\varepsilon'}(0))$ equal to one near
the origin, where $\varepsilon'$ is given by~\eqref{eqEpsilon'}. Set $U=\xi (V-V')$.
Clearly, $\supp U\subset\mathcal O_{\varepsilon'}(0)$ and
\begin{equation}\label{eqNonSmoothuNotPolyn4}
U\in \cW^\ell(K), \qquad U\notin\cW^{2m}(K^d)\quad\forall d>0.
\end{equation}
Moreover, by virtue of~\eqref{eqNonSmoothuNotPolyn3} and~\eqref{eq5.2'''}, we have
\begin{equation}\label{eqNonSmoothuNotPolyn5}
\{{\mathbf P}_j(y,D_y)U_j\}\in\mathcal W^0(K^\varepsilon),\quad \{{\mathbf
B}_{j\sigma\mu}(y,D_y)U-c_{j\sigma\mu}r^{s-m_{j\sigma\mu}}\}\in\mathcal
H_0^{\bm}(\gamma^\varepsilon).
\end{equation}

We note that, since $\{c_{j\sigma\mu}\}\in C_s$, the function
$c_{j\sigma\mu}r^{s-m_{j\sigma\mu}}$ either equals zero (which, in particular, holds for
$(j,\sigma,\mu)\in J_s$) or is a monomial of degree $s-m_{j\sigma\mu}$ (i.e., no greater
than $2m-m_{j\sigma\mu}-2$).

2. Consider the function $u(y)$ given by $u(y)=U_j(y'(y))$ for $y\in\mathcal
O_{\varepsilon'}(g_j)$ and $u(y)=0$ for $y\notin\mathcal O_{\varepsilon'}(\mathcal K)$,
where $y'\mapsto y(g_j)$ is the change of variables inverse to the change of variables
$y\mapsto y'(g_j)$ from Sec.~\ref{subsectStatement}. The function $u$ is the desired one.
Indeed, $u\notin W^{2m}(G)$ due to~\eqref{eqNonSmoothuNotPolyn4}. Furthermore, $\mathbf
B_{i\mu}^2u=0$ due to inequality~\eqref{eqSeparK23'} because $\supp u\subset \mathcal
O_{\varkappa_1}(\mathcal K)$. It follows from the equality $\mathbf B_{i\mu}^2u=0$ and
from relations~\eqref{eqNonSmoothuNotPolyn5} that the function $u$ satisfies the
following relations:
\begin{equation*}
\mathbf P(y,D_y)u\in L_2(G),\qquad \mathbf B_{i\mu}^0 u+\mathbf B_{i\mu}^1 u+\mathbf
B_{i\mu}^2 u=f_{i\mu}^1+f_{i\mu}^2,
\end{equation*}
where $f_{i\mu}^1$ is a polynomial\footnote{The function $f_{i\mu}^1$ (being written in
the system of coordinates originated at the point $g\in\overline{\Gamma_i}\cap\cK$)
either equals zero or is a monomial of degree  $s-m_{j\sigma\mu}$.} of degree no greater
than $ 2m-m_{i\mu}-2$ in a neighborhood of the point $g\in\overline{\Gamma_i}\cap\cK$ and
$f_{i\mu}^2\in H_0^{\mimu}(\Gamma_i)$.
\end{proof}

\begin{remark}
We remind that the space $\mathcal S^{\bm}(\partial G)$ was introduced in
Sec.~\ref{subsectNonHomogProb} in the case where the line $\Im\lambda=1-2m$ contains only
the proper eigenvalue $i(1-2m)$. In this case, it was proved in
Theorem~\ref{thUNonSmFNonConsist} that the smoothness of generalized solutions may
violate if the right-hand side $\{f_{i\mu}\}\in\cW^{\bm}(\pG)$ does not belong to
$\mathcal S^{\bm}(\partial G)$. Theorem~\ref{thNonSmoothuNotPolyn} shows that if
Condition~\ref{condEigenMonom} or Condition~\ref{condNoEigenMonom} fails, then the
smoothness of generalized solutions may violate even for  the right-hand side
$\{f_{i\mu}\}\in\cS^{\bm}(\pG)$.

On the other hand, it is on principle that the smoothness violation in
Theorem~\ref{thNonSmoothuNotPolyn} occurs for a nonzero (and even nonregular) right-hand
side $\{f_{i\mu}\}$. It can be proved that  if we confine ourselves with regular
right-hand sides, then Conditions~\ref{condEigenMonom} and~\ref{condNoEigenMonom} are not
necessary for the preservation of smoothness.
\end{remark}

\appendix
\section{Appendix}

This appendix is included for the reader's convenience. Here we collect some known
results on weighted spaces and   properties of nonlocal operators, which are most
frequently referred to in the main part of the paper.

\subsection{Some Properties of Sobolev and Weighted Spaces}
In this subsection, we formulate some results concerning properties of weighted spaces
introduced in Sec.~\ref{subsectStatement}. Set
$$
K=\{y\in{\mathbb R}^2:\ r>0,\ |\omega|<\omega_0\},
$$
$$
\gamma_{\sigma}=\{y\in\mathbb R^2:\ r>0,\ \omega=(-1)^\sigma \omega_0\}\qquad
(\sigma=1,2).
$$

\begin{lemma}[see Lemma~2.1 in~\cite{GurRJMP03}]\label{lAppL2.1GurRJMP03}
Let $\psi\in W^{k-1/2}(\gamma_\sigma)$ {\rm (}$\sigma=1$ or $2$, $k\ge2${\rm )},
$\supp\psi\subset\{0\le r\le\varepsilon\}$ for some $\varepsilon>0$, and
$$
\dfrac{d^s}{dr^s}\psi(0)=0,\quad s=0,\dots, k-2.
$$
Then $\psi\in H^{k-1/2}_{\delta}(\gamma_\sigma)$ for any $\delta>0$ and
$$
\|\psi\|_{\psi\in H^{k-1/2}_{\delta}(\gamma_\sigma)}\le c
\|\psi\|_{W^{k-1/2}(\gamma_\sigma)},
$$
where $c=c(\varepsilon,\delta)>0$ does not depend on $\psi$.
\end{lemma}

\begin{lemma}[see Lemma~3.3$'$ in~\cite{KondrTMMO67}]\label{lAppL3.3'Kondr}
Let a function $u\in H_a^k(K)$, where $k\ge0$ and $a\in\mathbb R$, be compactly
supported. Suppose that $p\in C^k(\overline{K})$ and $p(0)=0$. Then $pu\in H^k_{a-1}(K)$.
\end{lemma}

\begin{lemma}[see Lemma~4.20 in~\cite{KondrTMMO67}]\label{lAppL4.20Kondr}
The function $r^{i\lambda_0}\Phi(\omega)\ln^s r$, where $\Im\lambda_0=-(k-1)$ and $s\ge0$
is an integer, belongs to $W^k(K\cap\{|y|<1\})$ if and only if it is a homogeneous
polynomial in $y_1,y_2$ of degree $k-1$.
\end{lemma}

\begin{lemma}\label{lDense}
Let $f\in W^k(\mathbb R^2)$ and $D^\alpha f(0)=0$, $|\alpha|\le k-2$, if $k\ge2$. Then
there exists a sequence $f^n\in C_0^\infty(\mathbb R^2)$, $n=1,2,\dots$, such that
$f^n(y)=0$ in some neighborhood of the origin $($depending on $n)$ and $f^n\to f$ in
$W^{k}(\mathbb R^2)$.
\end{lemma}
\begin{proof} The proof  is analogous to that of Lemma 4.1 in~\cite{GurAdvDiffEq}. \end{proof}

\subsection{Nonlocal Problems in Plane Angles in Weighted Spaces}

In this subsection and in the next one, we formulate some properties of solutions of
problem~\eqref{eqPinK}, \eqref{eqBinK} in the spaces~\eqref{eqSpacesCal1}
and~\eqref{eqSpacesCal2}. First, we consider the case of weighted spaces.

For convenience, we rewrite this problem:
\begin{equation}\label{eqAppPBinK}
\begin{aligned}
\mathbf P_j(y,D_y)U_j&=F_j(y) & & (y\in K_j^\varepsilon),\\
 \mathbf
B_{j\sigma\mu}(y,D_y)U&=\Phi_{j\sigma\mu}(y) & & (y\in \gamma_{j\sigma}^\varepsilon),
\end{aligned}
\end{equation}
Along with problem~\eqref{eqAppPBinK}, we consider the following model problem in the
unbounded angles.
\begin{equation}\label{eqAppCalPBinK}
\begin{aligned}
\mathcal P_j(D_y)U_j&=F_j(y) & & (y\in K_j),\\
 \mathcal
B_{j\sigma\mu}(D_y)U&=\Phi_{j\sigma\mu}(y)& & (y\in \gamma_{j\sigma}).
\end{aligned}
\end{equation}

\begin{lemma}[see Lemma~2.3 in~\cite{GurMatZam05}]\label{lAppL2.3GurMatZam05}
Let a function  $U$ be a solution of problem~\eqref{eqAppPBinK}
{\rm(}or~\eqref{eqAppCalPBinK}{\rm)} such that
$$U_j\in W^{2m}(K_{j}^{D_\chi
\varepsilon}\setminus\overline{\cO_\delta(0)})\quad\forall\delta>0;\qquad U\in \mathcal
H_{a-2m}^0(K^{D_\chi \varepsilon}),
$$
where $D_\chi$ is given by~\eqref{eqd1d2} and $a\in\mathbb R$. Suppose that
$$
\{F_j\}\in \mathcal H_{a}^{0}(K^{\varepsilon}),\qquad \{\Phi_{j\sigma\mu}\}\in \mathcal
H_{a}^{\bm}(\gamma^\varepsilon).
$$
Then $ U\in\mathcal  H_{a}^{2m}(K^{\varepsilon}). $
\end{lemma}

Consider the asymptotics of solutions of problem~\eqref{eqAppCalPBinK}.
\begin{lemma}[see Lemma~2.1 in~\cite{GurPetr03}]\label{lAppL2.1GurPetr03}
The function
\begin{equation}\label{eqAppL2.1GurPetr03}
U=r^{i\lambda_0}\sum\limits_{l=0}^{l_0}\frac{\displaystyle 1}{\displaystyle l!}(i\ln
r)^l\varphi^{(l_0-l)}(\omega),
\end{equation}
is a solution of homogeneous problem~\eqref{eqAppCalPBinK} if and only if~$\lambda_0$ is
an eigenvalue of the operator~$\tilde{\cal L}(\lambda)$ and $\varphi^{(0)}, \dots,
\varphi^{(\varkappa-1)}$ is a Jordan chain corresponding to the eigenvalue~$\lambda_0;$
here $l_0\le\varkappa-1$.
\end{lemma}

Any solution of the kind~\eqref{eqAppL2.1GurPetr03} is called a \textit{power solution}.

\begin{lemma}[see Theorem~2.2 and Remark~2.2 in~\cite{GurPetr03}]\label{lAppTh2.2GurPetr03}
Let
$$\{F_j\}\in \mathcal H_a^0(K)\cap \mathcal H_{a'}^0(K),\quad
\{\Phi_{j\sigma\mu}\}\in \mathcal H_a^{\bm}(\gamma)\cap \mathcal H_{a'}^{\bm}(\gamma),$$
where $a>a'$. Suppose that the line $\Im\lambda=a'-1$ contains no eigenvalues of the
operator~$\tilde{\cal L}(\lambda)$. If $U$ is a solution of problem~\eqref{eqAppCalPBinK}
belonging to the space~$\mathcal H_a^{2m}(K)$, then
$$
U= \sum\limits_{n=1}^{n_0}\sum\limits_{q=1}^{J_n}\sum\limits_{l_0=0}^{\varkappa_{qn}-1}
c_n^{(l_0,q)}W_n^{(l_0,q)}(\omega, r)+U'.
$$
Here $\lambda_1, \dots, \lambda_{n_0}$ are eigenvalues of $\tilde{\cal L}(\lambda)$
located in the band $a'-1<\Im\lambda<a-1;$
$$
W_n^{(l_0,q)}(\omega, r)=r^{i\lambda_n}\sum\limits_{l=0}^{l_0}
   \frac{\displaystyle 1}{\displaystyle l!}(i\ln r)^l\varphi_n^{(l_0-l,q)}(\omega)
$$
are the power solutions of homogeneous problem~\eqref{eqAppCalPBinK}$;$
$$
 \{\varphi_n^{(0,q)}, \dots, \varphi_n^{(\varkappa_{qn}-1,q)}:\ q=1,\dots,J_n\}
$$
is a canonical system of Jordan chains of the operator~$\tilde{\cal L}(\lambda)$
corresponding to the eigenvalue~$\lambda_n;$ $c_n^{(m,q)}$ are some complex constants$;$
finally, $U'$ is a solution of problem~\eqref{eqAppCalPBinK} belonging to the space
$\mathcal H_{a'}^{2m}(K)$.
\end{lemma}

If the right-hand sides of problem~\eqref{eqAppCalPBinK} are of particular form, then
there exist solutions of particular form. Let
\begin{equation}\label{eqAppL4.3GurPetr03}
  F_j(\omega, r)=r^{i\lambda_0-2m}\sum\limits_{l=0}^M\frac{\displaystyle 1}{\displaystyle l!}(i\ln
  r)^l
  f_j^{(l)}(\omega),\quad
\Phi_{j\sigma \mu}(r)=r^{i\lambda_0-m_{j\sigma\mu}}
\sum\limits_{l=0}^M\frac{\displaystyle 1}{\displaystyle l!}(i\ln r)^l
\varphi_{j\sigma\mu}^{(l)},
\end{equation}
where $
  f_j^{(l)}\in L^2(-\omega_j,\omega_j),\ \varphi_{j\sigma\mu}^{(l)}\in{\mathbb
  C},\  \lambda_0\in\mathbb C.
$

If~$\lambda_0$ is an eigenvalue of the operator~$\tilde{\cal L}(\lambda)$, then denote
by~$\varkappa(\lambda_0)$ the greatest of partial multiplicities of this eigenvalue;
otherwise, set $\varkappa(\lambda_0)=0$.
\begin{lemma}[see Lemma~4.3 in~\cite{GurPetr03}]\label{lAppL4.3GurPetr03}
For problem~\eqref{eqAppCalPBinK} with right-hand side~$\{F_j,\Phi_{j\sigma\mu}\}$ given
by~\eqref{eqAppL4.3GurPetr03}, there exists a solution
 \begin{equation}\label{eqAppL4.3GurPetr03'}
U= r^{i\lambda_0}\sum\limits_{l=0}^{M+\varkappa(\lambda_0)} \frac{\displaystyle
1}{\displaystyle l!}(i\ln r)^l u^{(l)}(\omega),
 \end{equation}
where $u^{(l)}\in \prod\limits_j W^{2m}(-\omega_j,\omega_j)$. A solution of such a form
is unique if~$\varkappa(\lambda_0)=0$ {\rm(}i.e., $\lambda_0$ is not an eigenvalue
of~$\tilde{\cal L}(\lambda)${\rm)}. If~$\varkappa(\lambda_0)>0$, then the
solution~\eqref{eqAppL4.3GurPetr03'} is defined accurate to an arbitrary linear
combination of power solutions~\eqref{eqAppL2.1GurPetr03} corresponding to the
eigenvalue~$\lambda_0$.
\end{lemma}

The following result is a modification of Lemma~\ref{lAppTh2.2GurPetr03} for the case in
which the line $\Im\lambda=1-2m$ contains the unique eigenvalue $\lambda_0=i(1-2m)$ of
$\tilde{\cal L}(\lambda)$ and this eigenvalue is proper (see
Definition~\ref{defRegEigVal}).

\begin{lemma}[see Lemma~3.4 in~\cite{GurRJMP03}]\label{lAppL3.4GurRJMP03}
Let $U\in \mathcal H_a^{2m}(K)$, where $a>0$, be a solution of
problem~\eqref{eqAppCalPBinK} with right-hand side $\{F_j\}\in \mathcal H_a^0(K)\cap
\mathcal H_0^0(K)$, $\{\Phi_{j\sigma\mu}\}\in\mathcal H_a^{\bm}(\gamma)\cap\mathcal
H_0^{\bm}(\gamma)$. Suppose that the  band $1-2m\le\Im\lambda< a+1-2m$ contains only the
eigenvalue $\lambda_0=i(1-2m)$ of $\tilde{\mathcal L}(\lambda)$ and this eigenvalue is
proper. Then $D^\alpha U\in \mathcal H_0^0(K)$ for $|\alpha|=2m$.
\end{lemma}

\subsection{Nonlocal Problems in Plane Angles in Sobolev Spaces}

\begin{lemma}[see Lemma~2.4 and Corollary~2.1 in~\cite{GurRJMP03}]\label{lAppL2.4GurRJMP03}
Suppose~the\linebreak line~$\Im\lambda=1-2m$ contains no eigenvalues of $\tilde{\cal
L}(\lambda)$. Let
$$
\{\Phi_{j\sigma\mu}\}\in\mathcal W^{\bm}(\gamma^\varepsilon)\cap\mathcal
H_\delta^{\bm}(\gamma^\varepsilon)\qquad\forall\delta>0.
$$
Then there exists a compactly supported function $ V\in\mathcal W^{2m}(K)\cap\mathcal
H_{\delta}^{2m}(K), $ where $\delta>0$ is arbitrary, such that
$$
\{{\mathbf P}_{j}(y,D_y)V_j\}\in\mathcal H_0^0(K^\varepsilon),\qquad \{{\mathbf
B}_{j\sigma\mu}(y,D_y)V|_{\gamma_{j\sigma}^\varepsilon}-\Phi_{j\sigma\mu}\}\in \mathcal
H_0^{\bm}(\gamma^\varepsilon).
$$
\end{lemma}

Now we consider the situation where the line $\Im\lambda=1-2m$ contains the unique
eigenvalue $\lambda_0=i(1-2m)$ of $\tilde{\cal L}(\lambda)$ and it is proper (see
Definition~\ref{defRegEigVal}).

\begin{lemma}[see Lemma~3.3 and Corollary~3.1 in~\cite{GurRJMP03}]\label{lAppL3.3GurRJMP03}
Let the line\linebreak $\Im\lambda=1-2m$ contain only the unique eigenvalue
$\lambda_0=i(1-2m)$ of $\tilde{\cal L}(\lambda)$ and it is proper. Suppose that
$$
\{\Phi_{j\sigma\mu}\}\in\mathcal W^{\bm}(\gamma^\varepsilon)\cap\mathcal
H_\delta^{\bm}(\gamma^\varepsilon)\qquad\forall\delta>0
$$
and the functions $\Phi_{j\sigma\mu}$ satisfy the consistency
condition~\eqref{eqConsistencyZ}. Then there exists a compactly supported function $
V\in\mathcal  W^{2m}(K)\cap\mathcal H_{\delta}^{2m}(K), $ where $\delta>0$ is  arbitrary,
such that
$$
\{{\mathbf P}_{j}(y,D_y)V_j\}\in\mathcal H_0^0(K^\varepsilon),\qquad \{{\mathbf
B}_{j\sigma\mu}(y,D_y)V|_{\gamma_{j\sigma}^\varepsilon}-\Phi_{j\sigma\mu}\}\in \mathcal
H_0^{\bm}(\gamma^\varepsilon).
$$
\end{lemma}

\begin{lemma}[see Lemma~3.1 in~\cite{GurRJMP03}]\label{lAppL3.1GurRJMP03}
Let the line $\Im\lambda=1-2m$ contain only the proper eigenvalue $\lambda_0=i(1-2m)$ of
$\tilde{\cal L}(\lambda)$. Suppose that $U\in\mathcal W^{2m}(K)$ is a compactly supported
solution of problem~\eqref{eqAppPBinK} {\rm(}or~\eqref{eqAppCalPBinK}{\rm)} and $D^\alpha
U(0)=0$, $|\alpha|\le 2m-2$. Then the functions $\Phi_{j\sigma\mu}$ satisfy the
consistency condition~\eqref{eqConsistencyZ}.
\end{lemma}

\end{document}